\documentclass[11pt,a4paper]{article}
\usepackage{indentfirst}
\usepackage{fancyhdr}
\usepackage{color}
\usepackage{mathrsfs}
\usepackage{amsfonts,amssymb,amsmath,amsthm}
\usepackage{authblk}
\usepackage[unicode,pdftex]{hyperref}
\hypersetup{
	colorlinks = true,
	citecolor = green,
	anchorcolor = blue,
	linkcolor = blue}

\newtheorem{thm}{Theorem}[section]
\newtheorem{lemma}[thm]{Lemma}
\newtheorem{coro}[thm]{Corollary}
\newtheorem{prop}[thm]{Proposition}

\newtheoremstyle{rem}{10pt}{10pt}{\rmfamily}{}{\bfseries}{.}{.5em}{} 
\theoremstyle{rem}

\numberwithin{equation}{section} 

\title{Local well-posedness for the Schr\"{o}dinger-KdV system in $H^{s_1}\times H^{s_2}$, II}
\author[1]{Yingzhe Ban}
\author[2]{Jie Chen}
\author[3,1]{Ying Zhang}
\affil[1]{\scriptsize \textit{The Graduate School of China Academy of Engineering Physics, Beijing, 100088, P.R. China}}
\affil[2]{\scriptsize \textit{School of Sciences, Jimei University, Xiamen 361021, P.R. China}}
\affil[3]{\scriptsize \textit{Academy of Mathematics and Systems Science, CAS, Beijing 100190, P.R. China}}

\date{}

\begin{document}
	\maketitle
	\begin{abstract}
		In this paper, we continue the study of the local well-posedness theory for the Schr\"{o}dinger-KdV system in the Sobolev space $H^{s_1}\times H^{s_2}$. We show the local well-posedness in $H^{-3/16}\times H^{-3/4}$ for $\beta = 0$. Combining our work \cite{banchenzhang}, we also have the local well-posedness for $\max\{-3/4,s_1-3\}\leq s_2\leq \min\{4s_1,s_1+2\}$. The result is sharp by using the contraction mapping argument.
	\end{abstract}
	
	\section{Introduction}		
		We study the Cauchy problem for Schr\"{o}dinger-KdV system
		\begin{equation*}\label{model}\tag{S-KdV}
			\left\{
			\begin{aligned}
				&i\partial_t u+\partial_{xx}u = \alpha uv+\beta |u|^2 u,\\
				&\partial_t v+\partial_{xxx}v= \partial_x(\gamma|u|^2-v^2/2),\\
				&(u,v)|_{t = 0} = (u_0,v_0)\in H^{s_1}\times H^{s_2}.
			\end{aligned}
			\right.
		\end{equation*}
		In this paper we concern the local well-posedness theory for \eqref{model} with $\beta = 0$. The constants $\alpha,\gamma\neq 0$ in the model is not essential. We would assume $\alpha = \gamma = 1$ for brevity.
		
		This model appears in the study of resonant interaction between solitary wave of Langmuir and solitary wave of ion-acoustic. See Appert-Vaclavik \cite{appert1977dynamics}. It also appears in plasma physics and 
		a diatomic lattice system. See \cite{guo2010well,corcho2007well} and reference therein.
		
		We recall some early results for this model. Tsutsumi showed the well-posedness in $H^{s+1/2}\times H^s$, $s\in \mathbb{N}$ in \cite{tsutsumi1993well}. In \cite{guo1999well}, Guo--Miao showed the well-posedness in $H^s\times H^s$, $s\in \mathbb{N}$. Guo--Wang \cite{guo2010well} obtained the well-posedness in $H^{s_1}\times H^{-3/4}$, $s_1>-1/16$ and Wang--Cui  showed the well-posedness in $H^{s_1}\times H^{-3/4}$, $s_1>-3/16$ in \cite{wang2011cauchy} . See also \cite{bekiranov1997weak,corcho2007well,banchenzhang} for other results.
		
		In this paper, we obtain the following result.
		
		\begin{thm}
			\eqref{model} is local well-posed in $H^{-3/16}\times H^{-3/4}$.
		\end{thm}
		
		In fact combining the results in our work \cite{banchenzhang} we have
		\begin{coro}
			\eqref{model} is local well-posed in $H^{s_1}\times H^{s_2}$ for
			$$\max\{-3/4,s_1-3\}\leq s_2\leq  \min\{4s_1,s_1+2\}.$$
		\end{coro}
		
		Recall the proof of Theorem 1.1 in \cite{banchenzhang} one has
		\begin{thm}
			The data-to-solution mapping of \eqref{model} can not be $C^2$ except $\max\{-3/4,s_1-3\}\leq s_2\leq  \min\{4s_1,s_1+2\}$.
		\end{thm}
		Thus it is reasonable to say than our local well-posedness result is sharp by using contraction mapping argument.
		
		To show the local well-posedness in $H^{-3/16}\times H^{-3/4}$, there are two difficulties which come from the terms $\partial_x(|u|^2)$ and $\partial_x(v^2)$. In some sense, $H^{-3/16}\times H^{-3/4}$ is the ``double critical" space for this model. We need some techniques that are different from before.
		
		The paper is organized as follows. In Section \ref{workspaceeasyesti}, we rewrite \eqref{model} into a new form and construct the workspace. Then we show linear estimates, bilinear estimates related the Schr\"{o}dinger equation, and tirlinear estimate related the KdV equation. In Section \ref{estF}, we estimate $F[u_0]$ (see the definition \eqref{defiforF}). In Section \ref{birefforkd}, we give some refinement for the bilinear estimates related to the KdV equation and finally conclude the proof.
		
		\textbf{Notations.} . For $a, b\in \mathbb{R}^+$, $a \lesssim b$ means that there exists $C > 0$ such
		that $a \leq Cb$ and $a \sim b$ means $a \lesssim b \lesssim a$ where $C$ is a constant. We use $\langle \xi\rangle$ to denote $(1+\xi^2)^{1/2}$. Let $\varphi$ be an even, smooth function and $\varphi|_{[-1,1]} = 0$, $\varphi|_{[-2,2]^c} = 0$. $\varphi_N = \varphi(\cdot/N)$, $\psi_N = \varphi_N-\varphi_{N/2}$. We use inhomogeneous Littlewood-Paley projection: 
		$P_N:=\mathscr{F}^{-1}\psi_N\mathscr{F}$, $N\geq 2$; $P_1 = \mathscr{F}^{-1}\varphi\mathscr{F}$; $P_{>N} = \sum_{M>N}P_M$, $etc.$. We always use $N, L$ to denote a dyadic number larger than $1$. 
		
		Let $S_\lambda(t) = e^{i\lambda t\partial_{xx}}, K(t) = e^{-t\partial_{xxx}}$,
		$$\mathscr{A}_\lambda(f):=i\lambda\int_0^tS_\lambda(t-s)f(s)~ds,\quad \mathscr{B}(f):=\int_0^t K(t-s)\partial_x f(s)~ds.$$ 
		We always assume $0<\lambda\ll 1$ in this paper.
		We define the modulation decomposition $Q^{S_\lambda}_L$, $Q_L^K$ by $\mathscr{F}^{-1}_{\tau,\xi}\psi_L(\tau+\lambda\xi^2)\mathscr{F}_{t,x}$, $\mathscr{F}^{-1}_{\tau,\xi}\psi_L(\tau-\xi^3)\mathscr{F}_{t,x}$ respectively. Similarly, we define $Q_{>L}^{S_\lambda}$, $Q_{>L}^K$, $Q_1^{S_\lambda}$, and $Q_1^{K}$.
				
		\section{Function spaces and multilinear estimates}\label{workspaceeasyesti}
		To show the well-posedness of \eqref{model} with $s_2 = -3/4$, we use the argument in \cite{guo2010well}. By rescaling, we consider the equation
		\begin{equation*}
			\left\{
			\begin{aligned}
				&i\partial_t u+\lambda\partial_{xx}u = \lambda uv,\\
				&\partial_t v+\partial_{xxx}v = \partial_x(|u|^2-v^2),\\
				&(u,v)|_{t = 0} = (u_{0,\lambda},v_{0,\lambda})\in H^{-3/16}\times H^{-3/4}
			\end{aligned}
			\right.
		\end{equation*}
		where $0<\lambda\ll 1$, $\|u_{0,\lambda}\|_{H^{-3/16}}\lesssim \lambda^{21/16}$, $\|v_{0,\lambda}\|_{H^{3/4}}\lesssim \lambda^{3/4}$. Recall the definition of $U^p-V^p$ spaces in \cite{hadac2009well}. We define
		$$\|u\|_{X_\lambda}=\|N^{-3/16}P_Nu\|_{l^2_NU^2_{S_\lambda}},~~ \|v\|_{Y}=\|P_1v\|_{L_x^2L_t^\infty}+\|N^{-3/4}P_Nv\|_{l^2_{N\geq 2}U^2_K}.$$
		For the brevity, we still denote $(u_{0,\lambda},v_{0,\lambda})$ by $(u_0,v_0)$. The corresponding integral equation is
		\begin{equation*}
			\left\{
			\begin{aligned}
				u(t)&=S_\lambda(t)u_0-\mathscr{A}_\lambda(uv)(t),\\
				v(t)&=K(t)v_0+\mathscr{B}(|u|^2-v^2)(t).
			\end{aligned}
			\right.
		\end{equation*}
		We replace the function $u$ in second equation with the first one and obtain
		\begin{equation*}
			\left\{
			\begin{aligned}
				u(t)&=S_\lambda(t)u_0-i\mathscr{A}_\lambda(uv)(t),\\
				v(t)&=K(t)v_0+\mathscr{B}(|S_\lambda(t)u_0-\mathscr{A}_\lambda(uv)|^2-v^2)(t).
			\end{aligned}
			\right.
		\end{equation*}
		Let $\eta, \tilde{\eta}\in C_0^\infty$ with $\eta|_{[-1,1]} = 1$, $\tilde{\eta}|_{\mathrm{supp}(\eta)} = 1$. Define
		\begin{equation}\label{defiforF}
			F[u_0](t):=\eta(t)\mathscr{B}(|S_\lambda(t)u_0|^2)(t).
		\end{equation}
		In fact we can not show $F[u_0]\in Y$. Let $w(t) = v(t)-F[u_0](t)$. We consider the equation
		\begin{equation}\label{themodelmanu}
			\left\{
			\begin{aligned}
				u(t)&=\eta S_\lambda(t)u_0-\eta\mathscr{A}_\lambda(u(w+F[u_0]))(t),\\
				w(t)&=\eta K(t)v_0-\eta\mathscr{B}((w+F[u_0])^2)(t)+\eta\mathscr{B}(T_\lambda(u,w,u_0))(t)
			\end{aligned}
			\right.
		\end{equation}
		where $T_\lambda(u,w,u_0)= 2\mathrm{Im}(\mathscr{A}_\lambda(u(w+F[u_0]))\overline{S_\lambda(t)u_0})-|\mathscr{A}_\lambda(u(w+F[u_0]))|^2$.
		Define 
		\begin{equation}\label{normz}
			\|v\|_{Z}:=\|P_{\lesssim 1}v\|_{L_x^2L_t^\infty}+\|N^{-3/4}P_Nv\|_{l^2_{N\geq 2}V^2_K}+\|N^{1/4}P_Nv\|_{l^2_NL_x^\infty L_t^2}.
		\end{equation}
		Note that for $v$ supported on $[0,1]\times \mathbb{R}$, $\|v\|_Z\lesssim \|v\|_Y$ by Strichartz, maximal function, and local smoothing estimates. We would show $F[u_0]\in Z$ in Section \ref{estF}.
		
		
		Directly, one has the following estimates due to the maximal function estimate and the definition of $U^2$.
		\begin{lemma}\label{linear}
			$\|\eta S_{\lambda}(t)u_0\|_{X_{\lambda}}\lesssim \|u_0\|_{H^{-3/16}}$, $\|\eta K(t)v_0\|_{Y}\lesssim \|v_0\|_{H^{-3/4}}$.
		\end{lemma}
		Then, we show the bilinear estimate related to the Schr\"{o}dinger equation.
		\begin{lemma}\label{uvtou}
			$\|\eta\mathscr{A}_\lambda(uv)\|_{X_{\lambda}}\lesssim \lambda^{1/2} \|u\|_{X_{\lambda}}\|v\|_{Z}$.
		\end{lemma}
		\begin{proof}[\textbf{Proof.}]
		Firstly, for low frequency part of $v$ we have
		\begin{equation}\label{uvlow}
			\begin{aligned}
			\left|\int_{\mathbb{R}^2}P_{N} (uP_1v)\tilde{\eta}^2\bar{w}~dxdt\right|
			&\lesssim \sum_{N_1\sim N}\|\tilde{\eta}P_{N_1} u\|_{L_x^\infty L_t^2}\|P_1 v\|_{L_x^2L_t^\infty}\|\tilde{\eta}w\|_{L_{t,x}^2}\\
			&\lesssim (\lambda N)^{-1/2}\sum_{N_1\sim N}\|P_{N_1} u\|_{U^2_{S_\lambda}}\|v\|_{Z}\|w\|_{V^2_{S_\lambda}}.
			\end{aligned}
		\end{equation}
		By duality one has
		\begin{align*}
				\|\eta\mathscr{A}_\lambda(uP_1v)\|_{X_\lambda}&\lesssim \lambda\sum_{N}N^{-3/16}(\lambda N)^{-1/2}\sum_{N_1\sim N}\|P_{N_1} u\|_{U^2_{S_\lambda}}\|v\|_{Z}\\
				&\lesssim \lambda^{1/2}\|u\|_{X_\lambda}\|v\|_Z.
		\end{align*}
		Then by triangle inequality,
		\begin{align*}
			\|\eta\mathscr{A}_\lambda(uP_{>1}v)\|_{X_\lambda}\leq \sum_N\sum_{N_1}\sum_{N_2\geq 2}N^{-3/16}\|\eta P_{N}\mathscr{A}_\lambda(P_{N_1}uP_{N_2}v)\|_{U^2_{S_\lambda}}.
		\end{align*}
		If $N_2^2\nsim \lambda N_1$, for $\tau_1+\tau_2 = \tau$, $\xi_1+\xi_2 = \xi$, $|\xi_1|\sim N_1$, $|\xi_2|\sim N_2$, $|\xi|\sim N$ we have
		\begin{align*}
			|\tau_1+\lambda\xi_1^2|+|\tau_2-\xi^3_2|+|\tau+\lambda\xi^2|\gtrsim N_2\max\{N_2^2,\lambda N_1\}.
		\end{align*}
		Let $L = cN_2\max\{N_2^2,\lambda N_1\}$. Typically we need to estimate
		\begin{align*}
			I&=\int_{\mathbb{R}^2}Q_{>L}^{S_\lambda}P_{N_1}uP_{N_2}vP_N\bar{w}~dxdt,\\
			II&=\int_{\mathbb{R}^2}P_{N_1}uQ_{>L}^{K}P_{N_2}vP_N\bar{w}~dxdt
		\end{align*}
		where $u, v, w$ are supported on $[-1,1]\times \mathbb{R}$.
		By the Strichartz estimate one has
		$$\|u\|_{L_{t,x}^6}\lesssim \lambda^{-1/6}\|u\|_{U^2_{S_\lambda}}\lesssim \lambda^{-1/6}\|u\|_{V^2_{S_\lambda}}, ~\|u\|_{L_{t,x}^2}\lesssim \|u\|_{L_t^\infty L_x^2}\lesssim \|u\|_{V^2_{S_\lambda}}.$$
		Thus $\|u\|_{L_{t,x}^4}\leq \|u\|_{L_{t,x}^6}^{3/4}\|u\|_{L_{t,x}^2}^{1/4}\lesssim \lambda^{-1/8}\|u\|_{V^2_{S_\lambda}}$. Similarly,
		$$\|P_{N_2}v\|_{L_{t,x}^4}\lesssim N_2^{-1/8}\|v\|_{V^2_K}, \quad \|P_Nw\|_{L_{t,x}^4}\lesssim \lambda^{-1/8}\|w\|_{V^2_K}.$$
		By the H\"{o}lder inequality we have
		\begin{align*}
			|I|&\lesssim \|Q_{>L}^{S_\lambda}P_{N_1}u\|_{L_{t,x}^2}\|P_{N_2}v\|_{L_{t,x}^4}\|P_Nw\|_{L_{t,x}^4}\\
			&\lesssim L^{-1/2}\lambda^{-1/8}\|u\|_{V^2_{S_\lambda}}N_2^{-1/8}\|v\|_{V^2_K}\|w\|_{V^2_K}\\
			&\lesssim \lambda^{-1/8}N_2^{-5/8}\max\{N_2^2,\lambda N_1\}^{-1/2}\|u\|_{U^2_{S_\lambda}}\|v\|_{V^2_K}\|w\|_{V^2_K}.
		\end{align*}
		Also,
		\begin{align*}
			|II|\lesssim \lambda^{-1/4}N_2^{-1/2}\max\{N_2^2,\lambda N_1\}^{-1/2}\|u\|_{U^2_{S_\lambda}}\|v\|_{V^2_K}\|w\|_{V^2_K}.
		\end{align*}
		Thus by the duality one has
		\begin{equation}\label{nonreso}			
			\begin{aligned}
				&\quad\|\eta P_{N}\mathscr{A}_\lambda(P_{N_1}uP_{N_2}v)\|_{U^2_{S_\lambda}}\\
				&\lesssim \lambda\lambda^{-1/4}N_2^{-1/2}\max\{N_2^2,\lambda N_1\}^{-1/2}\|P_{N_1}u\|_{U^2_{S_\lambda}}\|P_{N_2}v\|_{V^2_K}\\
				&\lesssim \lambda^{3/4}N_1^{3/16}N_2^{1/4}\max\{N_2^2,\lambda N_1\}^{-1/2}\|u\|_{X_\lambda}\|v\|_{Z}.
			\end{aligned}
		\end{equation}
		Let $N_{\max}, N_{\mathrm{med}}$ be the maximum, medium of $N_1,N_2,N$. Note that
		$$\sum_{N_{\max}\sim N_{\mathrm{med}}}N^{-3/16}\lambda^{3/4}N_1^{3/16}N_2^{1/4}\max\{N_2^2,\lambda N_1\}^{-1/2}\lesssim \lambda^{3/4}\log\lambda^{-1}.$$
		Thus we obtain
		$$\sum_{N_2\geq 2, N_2^2\nsim \lambda N_1,N_1,N}N^{-3/16}\|\eta P_{N}\mathscr{A}_\lambda(P_{N_1}uP_{N_2}v)\|_{U^2_{S_\lambda}}\lesssim \frac{\log\lambda^{-1}}{\lambda^{-3/4}}\|u\|_{X_\lambda}\|v\|_Z.$$
		If $N_2^2\sim \lambda N_1\sim \lambda N$, by Lemma 4.12, (4.8) in \cite{banchenzhang}, we have
		\begin{align*}
			\left|\int_{\mathbb{R}^2} P_{N_1}uP_{N_2}vP_{N}\bar{w}~dxdt\right|\lesssim N_2^{-1}\|u\|_{U^2_{S_\lambda}}\|v\|_{U^2_K}\|w\|_{U^2_{S_\lambda}}.
		\end{align*}
		On the other hand, by H\"{o}lder's inequality and Strichartz estimates, one has
		\begin{align*}
			\left|\int_{\mathbb{R}^2} P_{N_1}uP_{N_2}vP_{N}\bar{w}~dxdt\right|&\lesssim \|u\|_{L_{t,x}^3}\|P_{N_2}v\|_{L_{t,x}^3}\|w\|_{L_{t,x}^3}\\
			&\lesssim \|u\|_{L_{t}^{12}L_{x}^3}\|P_{N_2}v\|_{L_{t}^{12}L_{x}^3}\|w\|_{L_{t}^{12}L_{x}^3}\\
			&\lesssim \lambda^{-1/6}N_2^{-1/12}\|u\|_{U^{12}_{S_\lambda}}\|v\|_{U^{12}_K}\|w\|_{U^{12}_{S_\lambda}}.
		\end{align*}
		Then by the interpolation (cf.  Proposition 2.20 in \cite{hadac2009well}), we have
		\begin{align*}
			\left|\int_{\mathbb{R}^2} P_{N_1}uP_{N_2}vP_{N}\bar{w}~dxdt\right|\lesssim N_2^{-1}\lambda^{-1/6}(\log N_2)^3\|u\|_{U^2_{S_\lambda}}\|v\|_{V^2_K}\|w\|_{V^2_{S_\lambda}}.
		\end{align*}
		By the duality one has
		\begin{equation}\label{reslab}
			\begin{aligned}
				\|\eta P_N\mathscr{A}_\lambda(P_{N_1}uP_{N_2}v)\|_{U^2_{S_\lambda}}&\lesssim \lambda N_2^{-1}\lambda^{-1/6}(\log N_2)^3\|P_{N_1}u\|_{U^2_{S_\lambda}}\|P_{N_2}v\|_{V^2_K}\\
				&\lesssim \lambda^{5/6}N_2^{-1/4}N_1^{3/16}(\log N_2)^3\|u\|_{X_\lambda}\|v\|_{Z}.
			\end{aligned}
		\end{equation}
		Note that
		$$\sum_{N_2^2\sim \lambda N_1\sim \lambda N}N^{-3/16}N_2^{-1/4}N_1^{3/16}\lambda^{5/6}(\log N_2)^3\lesssim \lambda^{5/6}.$$
		Overall one has
		\begin{align*}
			&\quad\sum_N\sum_{N_1}\sum_{N_2\geq 2}N^{-3/16}\|\eta P_{N}\mathscr{A}_\lambda(P_{N_1}uP_{N_2}v)\|_{U^2_{S_\lambda}}\lesssim \lambda^{1/2}\|u\|_{X_\lambda}\|v\|_{Z}.
		\end{align*}
		We finish the proof of this lemma.
		\end{proof}
		Next we show the trilinear estimate related to the KdV equation.
		\begin{lemma}\label{uvwtov}
			$\|\eta\mathscr{B}(\mathscr{A}_\lambda(uv)\bar{w})\|_{Y}\lesssim \lambda^{-1/2} \|u\|_{X_{\lambda}}\|v\|_{Z}\|w\|_{X_{\lambda}}$.
		\end{lemma}
		\begin{proof}[\textbf{Proof}]
			Without loss of the generality we would assume that $u, v, w$ are supported on $[-1,1]\times \mathbb{R}$. For the low frequency we have
			\begin{align*}
				\|\eta P_1\mathscr{B}(\mathscr{A}_\lambda(uv)\bar{w})\|_{L_x^2 L_t^\infty}&\lesssim \|P_1(\mathscr{A}(uv)\bar{w})\|_{L_{x}^2 L_t^1}\\
				&\lesssim \sum_{N_1\sim N_2}\|P_1(P_{N_1}\mathscr{A}_\lambda(uv)P_{N_2}\bar{w})\|_{L_{x}^2 L_t^1}\\
				&\lesssim \sum_{N_1\sim N_2}\|\tilde{\eta}P_{N_1}\mathscr{A}_\lambda(uv)\|_{L_{t,x}^2}\|P_{N_2}w\|_{L_{x}^\infty L_t^2}\\
				&\lesssim \sum_{N_1\sim N_2}\|\tilde{\eta}P_{N_1}\mathscr{A}_\lambda(uv)\|_{U^2_{S_\lambda}}(\lambda N_2)^{-1/2}\|P_{N_2}w\|_{U^2_{S_\lambda}}\\
				&\lesssim \lambda^{-1/2}\sum_{N_1\sim N_2}N_1^{3/16}N_2^{-5/16}\|\tilde{\eta}\mathscr{A}_\lambda(uv)\|_{X_\lambda}\|w\|_{X_\lambda}\\
				&\lesssim \lambda^{-1/2}\|\tilde{\eta}\mathscr{A}_\lambda(uv)\|_{X_\lambda}\|w\|_{X_\lambda}.
			\end{align*}
			By Lemma \ref{uvtou} we obtain
			\begin{equation}\label{lowoutput}
				\begin{aligned}
					\|\eta P_1\mathscr{B}(\mathscr{A}_\lambda(uv)\bar{w})\|_{L_x^2 L_t^\infty}\lesssim  \|u\|_{X_\lambda}\|v\|_Z\|w\|_{X_\lambda}.
				\end{aligned}
			\end{equation}
			Since $\|P_N Q_L^Kv\|_{L_x^\infty L_{t}^2}\lesssim N^{-1}L^{1/2}\|P_N Q_L^Kv\|_{L_{t,x}^2}$, by the interpolation one has
			$$\|P_NQ_L^K v\|_{L_x^{2/(1-\theta)}L_t^2}\lesssim (N^{-1}L^{1/2})^{\theta}\|P_NQ_L^K v\|_{L^2_{t,x}},\quad \forall~0<\theta<1.$$
			Thus ($v$ is supported on $[-1,1]\times \mathbb{R}$)
			\begin{align*}
				\|P_N v\|_{L_x^{2/(1-\theta)}L_t^2} & \lesssim \sum_L (N^{-1}L^{1/2})^{\theta}\|P_NQ_L^K v\|_{L^2_{t,x}}\\
					&\lesssim \sum_L (N^{-1}L^{1/2})^{\theta} L^{-1/2}\|P_Nv\|_{V^2_K}\lesssim N^{-\theta}\|P_Nv\|_{V^2_K}.
			\end{align*}
			If $N_2\lesssim N$ (which implies $N_1\lesssim N$), by the duality, Minkowski, and Bernstein inequalities one has
			\begin{align*}
				&\quad\sum_{N_1,N_2\lesssim N}\|\eta P_N\mathscr{B}(\mathscr{A}_\lambda(P_{N_1}uP_{1}v)P_{N_2}\bar{w})\|_{U^2_K}\\
				&\lesssim \sum_{N_1,N_2\lesssim N} NN^{-\theta}\|P_{N_1}\mathscr{A}_\lambda(uP_1 v)P_{N_2}w\|_{L_x^{2/(1+\theta)}L_t^2}\\
				&\lesssim \sum_{N_1,N_2\lesssim N}N^{1-\theta}\|\tilde{\eta}P_{N_1}\mathscr{A}_\lambda(uP_1 v)\|_{L_x^{2/\theta}L_t^2}\|\tilde{\eta}P_{N_2}w\|_{L_x^2 L_t^\infty}\\
				&\lesssim N^{3/2-\theta}\sum_{N_1,N_2\lesssim N}(\lambda N_1)^{(\theta-1)/2}\|\tilde{\eta}P_{N_1}\mathscr{A}_\lambda(uP_1 v)\|_{U^2_{S_\lambda}}\|P_{N_2}w\|_{U^2_{S_\lambda}}.
			\end{align*}
			By \eqref{uvlow} one has $\|\tilde{\eta}P_{N_1}\mathscr{A}_\lambda(uP_1 v)\|_{U^2_{S_\lambda}}\lesssim \lambda (\lambda N_1)^{-1/2}N_1^{3/16}\|u\|_{X_\lambda}\|v\|_{Z}$. Thus we can control the former term by
			\begin{align*}
				 N^{3/2-\theta}\sum_{N_1, N_2\lesssim N}(\lambda N_1)^{(\theta-1)/2}\lambda(\lambda N_1)^{-1/2}N_1^{3/16}\|u\|_{X_\lambda}\|v\|_{Z}N_2^{3/16}\|w\|_{X_\lambda},
			\end{align*}
			By choosing $15/16<\theta<1$, then
			$$\lambda\sum_{N}N^{-3/4}N^{3/2-\theta}\sum_{N_1, N_2\lesssim N}(\lambda N_1)^{(\theta-1)/2}(\lambda N_1)^{-1/2}N_1^{3/16}N_2^{3/16}\lesssim \lambda^{\theta/2}.$$
			We obtain the desired estimate since we always assume  $0<\lambda\ll 1$.
			
			If $N_2\gg N$ (which implies $N_1\sim N_2$) similar to the former argument, we have $\|P_{N}v\|_{L_{t,x}^{30/7}}\lesssim N^{-2/15}\|v\|_{V^2_K}$. Then by the duality one has
			\begin{align*}
				&\quad\sum_{N_1\sim N_2\gg N}\|\eta P_N\mathscr{B}(\mathscr{A}_{\lambda}(P_{N_1}uP_{1}v)P_{N_2}\bar{w})\|_{U^2_K}\\
				&\lesssim \sum_{N_1\sim N_2\gg N} NN^{-2/15}\|\mathscr{A}_{\lambda}(P_{N_1}uP_1 v)P_{N_2}w\|_{L_{t,x}^{30/23}}\\
				&\lesssim \sum_{N_1\sim N_3\gg N} N^{13/15}\|\tilde{\eta}\mathscr{A}_{\lambda}(P_{N_1}uP_1 v)\|_{L_{t,x}^2 }\|P_{N_2}w\|_{L_{t,x}^2}^{3/10}\|P_{N_2}w\|_{L_{t,x}^6}^{7/10}\\
				&\lesssim \lambda^{53/60}\sum_{N_1\sim N_2\gg N} N^{13/15}\|P_{N_1}u\|_{L^\infty_xL_t^2}\|P_1v\|_{L_x^2 L_t^\infty}\|P_{N_2}w\|_{U^2_{S_\lambda}}\\
				&\lesssim \lambda^{53/60}N^{13/15}\sum_{N_1\sim N_2\gg N}N_2^{3/16}(\lambda N_1)^{-1/2}N_1^{3/16}\|u\|_{X_\lambda}\|v\|_Z\|w\|_{X_\lambda}\\
				&\lesssim \lambda^{23/60}N^{89/120}\|u\|_{X_\lambda}\|v\|_Z\|w\|_{X_\lambda}.
			\end{align*}
			Note that $89/120<3/4$. Thus we have
			\begin{equation}\label{lowofv}
				\begin{aligned}
					&\quad \sum_{N\geq 2}N^{-3/4}\sum_{N_1\sim N_2\gg N}\|\eta P_N\mathscr{B}(\mathscr{A}_{\lambda}(P_{N_1}uP_{1}v)P_{N_2}\bar{w})\|_{U^2_K}\\
					&\lesssim \lambda^{23/60}\|u\|_{X_\lambda}\|v\|_Z\|w\|_{X_\lambda}.
				\end{aligned}
			\end{equation}
			By \eqref{nonreso}, for $M_2^2\nsim \lambda M_1$ we have
			\begin{align*}
				\|\eta P_{N_1}\mathscr{A}_{\lambda}(P_{M_1}uP_{M_2}v)\|_{U^2_{S_\lambda}}\lesssim \frac{\lambda^{3/4}M_1^{3/16}M_2^{1/4}}{\max\{M_2^2,\lambda M_1\}^{1/2}}\|u\|_{X_\lambda}\|v\|_{Z}.
			\end{align*}
			By \eqref{reslab}, for $M_2^2\sim \lambda M_1$ we have
			\begin{align*}
				\|\eta P_{N_1}\mathscr{A}_{\lambda}( P_{M_1}uP_{M_2}v)\|_{U^2_{S_\lambda}}\lesssim\lambda^{5/6} M_2^{-1/4}M_1^{3/16}(\log(M_2/\lambda))^3\|u\|_{X_\lambda}\|v\|_{Z}.
			\end{align*}
			Thus,
			\begin{equation}\label{refinuvtou}
				\begin{aligned}
					&\quad\|\eta P_{N_1}(\mathscr{A}_{\lambda}(u P_{>1}v))\|_{U^2_{S_\lambda}}\\
					&\lesssim \sum_{M_2^2\nsim \lambda M_1} \frac{\lambda^{3/4}M_1^{3/16}M_2^{1/4}}{\max\{M_2^2,\lambda M_1\}^{1/2}}\|u\|_{X_\lambda}\|v\|_{Z}\\
					&\quad +\sum_{M_2^2\sim \lambda M_1\sim \lambda N_1}\lambda^{5/6}M_2^{-1/4}M_1^{3/16}(\log(M_2/\lambda))^3\|u\|_{X_\lambda}\|v\|_{Z}\\
					&\lesssim (\lambda^{9/16}+\lambda^{17/24}N_1^{1/16}(\log N_1)^3)\|u\|_{X_\lambda}\|v\|_{Z}.
				\end{aligned}
			\end{equation}
			Let $\tilde{u} = \eta\mathscr{A}_{\lambda}(u P_{>1} v)$. Now we estimate $\|\eta P_N\mathscr{B}(P_{N_1}\tilde{u}P_{N_2}\bar{w})\|_{U^2_K}$.	

			If $N^2\ll \lambda N_1$ by Lemma 4.12, (4.7) in \cite{banchenzhang} one has
			$$\left|\int_{\mathbb{R}^2}\partial_x(P_{N_1}\tilde{u}P_{N_2}\bar{w})P_N\bar{\tilde{v}}~dxdt\right|\lesssim \lambda^{-1}N_1^{-1/2}\|\tilde{u}\|_{U^2_{S_\lambda}}\|w\|_{U^2_{S_\lambda}}\|v\|_{U^2_K}.$$
			On the other hand by the H\"{o}lder inequality we also have
			\begin{align*}
				\left|\int_{\mathbb{R}^2}\partial_x(P_{N_1}\tilde{u}P_{N_2}\bar{w})P_N\bar{\tilde{v}}~dxdt\right|&\lesssim N\|P_N(P_{N_1}\tilde{u}P_{N_2}\bar{w})\|_{L^2_{t,x}}\|\tilde{\eta}P_N\tilde{v}\|_{L^2_{t,x}}\\
				&\lesssim \lambda ^{-1/2}N^{1/2}\|\tilde{u}\|_{U^2_{S_\lambda}}\|w\|_{U^2_{S_\lambda}}\|v\|_{U^p_K},\quad p>2.
			\end{align*}
			Then by the interpolation theorem in \cite{hadac2009well} one has
			$$\left|\int_{\mathbb{R}^2}\partial_x(P_{N_1}\tilde{u}P_{N_2}\bar{w})P_N\bar{\tilde{v}}~dxdt\right|\lesssim \lambda^{-1}N_1^{-1/2}\log (N_1/\lambda)\|\tilde{u}\|_{U^2_{S_\lambda}}\|w\|_{U^2_{S_\lambda}}\|v\|_{V^2_K}.$$
			Combining Lemma \ref{uvtou}  we have
			\begin{equation}\label{N1large}
				\begin{aligned}
					&\quad\sum_{N\geq 2}N^{-3/4}\sum_{N_1\sim N_2\gg N^2/\lambda}\|\eta P_N\mathscr{B}(P_{N_1}\tilde{u}P_{N_2}\bar{w})\|_{U^2_K}\\
					&\lesssim \sum_{N\geq 2}N^{-3/4}\sum_{N_1\sim N_2\gg N^2/\lambda}\lambda^{-1}N_1^{-1/2}\log (N_1/\lambda)\|P_{N_1}\tilde{u}\|_{U^2_{S_\lambda}}\|P_{N_2}w\|_{U^2_{S_\lambda}}\\
					&\lesssim \sum_{N_1\sim N_2\gg 1/\lambda}\lambda^{-1}N_1^{-1/2}\log (N_1/\lambda)\|P_{N_1}\tilde{u}\|_{U^2_{S_\lambda}}\|P_{N_2}w\|_{U^2_{S_\lambda}}\\
					&\lesssim \lambda^{-7/8} \|\tilde{u}\|_{X_\lambda}\|w\|_{X_\lambda}\lesssim \lambda^{-3/8}\|u\|_{X_\lambda}\|v\|_Z\|w\|_{X_\lambda}.
				\end{aligned}
			\end{equation}
			Similarly if $N^2\gg \lambda N_1$ by Lemma 4.12, (4.6) in \cite{banchenzhang} and the interpolation one has
			$$\|\eta P_N\mathscr{B}(P_{N_1}\tilde{u}P_{N_2}\bar{w})\|_{U^2_K}\lesssim \lambda^{-1/2}N^{-1}\log (N/\lambda)\|\tilde{u}\|_{U^2_{S_\lambda}}\|w\|_{U^2_{S_\lambda}}.$$
			Combining Lemma \ref{uvtou}  we have
			\begin{equation}\label{Nrelarge}
				\begin{aligned}
					&\quad\sum_{N\geq 2}N^{-3/4}\sum_{N_1, N_2\ll N^2/\lambda}\|\eta P_N\mathscr{B}(P_{N_1}\tilde{u}P_{N_2}\bar{w})\|_{U^2_K}\\
					&\lesssim \sum_{N\geq 2}N^{-3/4}\sum_{N_1,N_2\ll N^2/\lambda}\lambda^{-1/2}N^{-1}\log (N/\lambda)\|P_{N_1}\tilde{u}\|_{U^2_{S_\lambda}}\|P_{N_2}w\|_{U^2_{S_\lambda}}\\
					&\lesssim \sum_{N\geq 2}N^{-3/4}\sum_{N_1,N_2\ll N^2/\lambda}\lambda^{-1/2}N^{-1}\log (N/\lambda)(N_1N_2)^{3/16}\|\tilde{u}\|_{X_\lambda}\|w\|_{X_\lambda}\\
					&\lesssim \lambda^{-7/8}\log(1/\lambda) \|\tilde{u}\|_{X_\lambda}\|w\|_{X_\lambda}\lesssim \lambda^{-3/8}\log(1/\lambda)\|u\|_{X_\lambda}\|v\|_Z\|w\|_{X_\lambda}.
				\end{aligned}
			\end{equation}			
			If $N^2\sim \lambda N_1$ similar to the former argument by Lemma 4.12, (4.8) one has
			$$\|\eta P_N\mathscr{B}(P_{N_1}\tilde{u}P_{N_2}\bar{w})\|_{U^2_K}\lesssim \log(N/\lambda)\|\tilde{u}\|_{U^2_{S_\lambda}}\|w\|_{U^2_{S_\lambda}}.$$
			Then by \eqref{refinuvtou} we have
			\begin{equation}\label{reson}
				\begin{aligned}
					&\quad\sum_{N\geq 2}N^{-3/4}\sum_{N_1\sim N_2\sim N^2/\lambda}\|\eta P_N\mathscr{B}(P_{N_1}\tilde{u}P_{N_2}\bar{w})\|_{U^2_K}\\
					&\lesssim \sum_{N\geq 2}N^{-3/4}\sum_{N_1\sim N_2\sim N^2/\lambda} \log(N/\lambda)\|P_{N_1}\tilde{u}\|_{U^2_{S_\lambda}}\|P_{N_2}w\|_{U^2_{S_\lambda}}\\
					&\lesssim \sum_{N\geq 2}N^{-3/4}\sum_{N_1\sim N_2\sim N^2/\lambda} \log(N/\lambda)N_2^{3/16}\\
					&\hspace{80pt}\cdot(\lambda^{9/16}+\lambda^{17/24}N_1^{1/16}(\log N_1)^3)\|u\|_{X_\lambda}\|v\|_{Y}\|w\|_{X_\lambda}\\
					&\lesssim \lambda^{3/8}\|u\|_{X_\lambda}\|v\|_Z\|w\|_{X_\lambda}.
				\end{aligned}
			\end{equation}
			Combining \eqref{N1large}, \eqref{Nrelarge}, and \eqref{reson} we conclude the proof.
		\end{proof}
		
		\section{The estimates for \texorpdfstring{$F[u_0]$}{F[u0]}}\label{estF}
		In this section, we show some estimates for $F[u_0]$. For the nonresonant case, we show that it belongs to $U^2$ type space. See Lemma \ref{estiforF}. Also we show $F[u_0]\in Z$.  
		\begin{lemma}\label{estiforF}
			Recall the definition of $F[u_0]$ by \eqref{defiforF}. We have
			\begin{equation*}
				\|N^{-3/4}P_{N}(F[u_0]-F[P_{\sim N^2/\lambda}u_0])\|_{l^2_NU^2_K}\lesssim \lambda^{-1}\|u_0\|_{H^{-3/16}}^2.
			\end{equation*}
		\end{lemma}
		\begin{proof}[\textbf{Proof}]
			If $N\lesssim 1$, we claim that
			\begin{equation}\label{lowfreq}
				\left|\int_{\mathbb{R}^2}S_\lambda(t)P_{N_1}u_0\overline{S_\lambda(t)P_{N_2}u_0}P_N\bar{v}~dxdt\right|\lesssim N_1^{-3/8}\|u_0\|_{L^2}^2\|v\|_{V^2_K}
			\end{equation}
			where $v$ is supported on $[-1/2,1/2]\times \mathbb{R}$. Let $q>2$. By the H\"{o}lder inequality, local smoothing, maximal function, and Strichartz estimates one has
			\begin{align*}
				&\quad \left|\int_{\mathbb{R}^2}S_\lambda(t)P_{N_1}u_0\overline{S_\lambda(t)P_{N_2}u_0}P_N\bar{v}~dxdt\right|\\
				&\lesssim \|\tilde{\eta}S_{\lambda}(t)P_{N_1}u_0\|_{L_x^{2q/(q-2)} L_t^2}\|\tilde{\eta}S_{\lambda}(t)P_{N_2}u_0\|_{L_{t,x}^2}\|P_{N}v\|_{L_x^qL_t^\infty}\\
				&\lesssim \|\tilde{\eta}S_{\lambda}(t)P_{N_1}u_0\|_{L_x^{\infty} L_t^2}^{2/q}\|u_0\|^{1+(q-2)/q}_{L_{x}^2}\|v\|_{V^2_K}\\
				&\lesssim (\lambda N_1)^{-1/q}\|u_0\|_{L^2}^2\|v\|_{V^2_K}.
			\end{align*}
			By choosing $q = 8/3$ we obtain \eqref{lowfreq}. Then by the duality one has
			\begin{align*}
				&\quad \sum_{N\lesssim 1}\sum_{N_1\sim N_2}N^{-3/4}\left\|\eta P_N\mathscr{B}(S_\lambda(t)P_{N_1}u_0\overline{S_\lambda(t)P_{N_2}u_0})\right\|_{U^2_K}\\
				&\lesssim \sum_{N_1\sim N_2}(\lambda N_1)^{-3/8}\|P_{N_1}u_0\|_{L^2}\|P_{N_2}u_0\|_{L^2}\lesssim \lambda^{-3/8}\|u_0\|_{H^{3/16}}^2.
			\end{align*}
			If $\lambda N_1\ll N^2$, by Lemma 4.12 in \cite{banchenzhang} one has
			\begin{equation*}
				\left|\int_{\mathbb{R}^2} \partial_x(P_{N_1}S_\lambda(t)u_0\overline{S_\lambda(t) P_{N_2}u_0}) P_N\bar{v}~dxdt\right|\lesssim \lambda^{-1/2}N^{-1}\|u_0\|_{L^2}^2\|v\|_{U^2_K}.
			\end{equation*}
			By the H\"{o}lder inequality and the Strichartz estimate one also has
			\begin{align*}
				&\quad\left|\int_{\mathbb{R}^2} \partial_x(P_{N_1}S_\lambda(t)u_0\overline{S_\lambda(t) P_{N_2}u_0}) P_N\bar{v}~dxdt\right|\\
				&\lesssim \|\tilde{\eta}P_{N_1}S_\lambda(t)u_0\|_{L^3_{t,x}}\|\tilde{\eta}P_{N_2}S_\lambda(t)u_0\|_{L^2_{t,x}}\|\partial_x P_N v\|_{L^6_{t,x}}\\
				&\lesssim \|P_{N_1}S_\lambda(t)u_0\|^{1/2}_{L^6_{t,x}}\|u_0\|^{3/2}_{L^2}N^{5/6}\|v\|_{U^6_K}\\
				&\lesssim \lambda^{-1/12}N^{5/6}\|u_0\|_{L^2}^2\|v\|_{U^6_K}
			\end{align*}
			where $v$ is supported on $[-1/2,1/2]\times \mathbb{R}$. Then by the interpolation in \cite{hadac2009well} we have
			\begin{equation*}
				\left\|\eta P_N\mathscr{B}(P_{N_1}S_\lambda(t)u_0\overline{S_\lambda(t) P_{N_2}u_0}) \right\|_{U^2_K}\lesssim \lambda^{-1/2}N^{-1}\log N\|u_0\|_{L^2}^2.
			\end{equation*}
			Since $\max\{N_1, N_2, N\}\sim \mathrm{med}\{N_1,N_2,N\}$, thus
			\begin{align*}
				&\quad\sum_{N\gg 1}\sum_{\lambda N_1\ll N^2}N^{-3/4}\left\|\eta P_N\mathscr{B}(P_{N_1}S_\lambda(t)u_0\overline{S_\lambda(t) P_{N_2}u_0}) \right\|_{U^2_K}\\
				&\lesssim \sum_{N\gg 1}\sum_{\lambda \max\{N_1,N_2\}\ll N^2}N^{-3/4}\lambda^{-1/2}N^{-1}\log N\|P_{N_1}u_0\|_{L^2}\|P_{N_2}u_0\|_{L^2}\\
				&\lesssim \lambda^{-7/8}\|u_0\|_{H^{-3/16}}^2.
			\end{align*}
			If $\lambda N_1\gg N^2$, by Lemma 4.12 in \cite{banchenzhang} one has
			\begin{equation*}
				\left|\int_{\mathbb{R}^2} \partial_x(P_{N_1}S_\lambda(t)u_0\overline{S_\lambda(t) P_{N_2}u_0}) P_N\bar{v}~dxdt\right|\lesssim \lambda^{-1}N_1^{-1/2}\|u_0\|_{L^2}^2\|v\|_{U^2_K}.
			\end{equation*}
			By similar argument as $\lambda N_1\ll N^2$ and  the interpolation in  \cite{hadac2009well} we have
			\begin{equation*}
				\left\|\eta P_N\mathscr{B}(P_{N_1}S_\lambda(t)u_0\overline{S_\lambda(t) P_{N_2}u_0}) \right\|_{U^2_K}\lesssim \lambda^{-1} N_1^{-1/2}\log N_1\|u_0\|_{L^2}^2.
			\end{equation*}
			Thus,
			\begin{align*}
				&\quad\sum_{N\gg 1}\sum_{\lambda N_1\gg N^2}N^{-3/4}\left\|\eta P_N\mathscr{B}(P_{N_1}S_\lambda(t)u_0\overline{S_\lambda(t) P_{N_2}u_0}) \right\|_{U^2_K}\\
				&\lesssim \sum_{N\gg 1}\sum_{N_1\sim N_2}N^{-3/4}\lambda^{-1}N_1^{-1/2}\log N_1\|P_{N_1}u_0\|_{L^2}\|P_{N_2}u_0\|_{L^2}\\
				&\lesssim \lambda^{-1}\|u_0\|_{H^{-3/16}}^2.
			\end{align*}
			We conclude the proof.
		\end{proof}
		Although we can not show the desired estimate for $F[P_{\sim N^2/\lambda}u_0]$ in $U^2_K$, one has the following estimate.
		\begin{lemma}\label{estiforF2}
			$\|N^{-3/4}P_{N}F[P_{\sim N^2/\lambda}u_0]\|_{l^2_NV^2_K}\lesssim \lambda^{-3/8}\|u_0\|_{H^{-3/16}}^2$.
		\end{lemma}
		\begin{proof}[\textbf{Proof}]
			By Lemma 4.12 in \cite{banchenzhang} one has
			\begin{equation*}
				\left|\int_{\mathbb{R}^2}\partial_x(|S_\lambda(t)P_{\sim N^2/\lambda}u_0|^2)P_N \bar{v}~dxdt\right|\lesssim \|u_0\|_{L^2}^2\|v\|_{U^2_K}
			\end{equation*}
			where $v$ is supported on $[-1,1]\times \mathbb{R}$. Then by the duality one has
			\begin{align*}
				\|N^{-3/4}P_{N}F[P_{\sim N^2/\lambda}u_0]\|_{l^2_NV^2_K}& \lesssim \|N^{-3/4}\|P_{\sim N^2/\lambda}u_0\|_{L^2}^2\|_{l^2_N}\\
				&\lesssim \|(\lambda M)^{-3/16}P_Mu_0\|^2_{l^2_ML^2}\\
				&\lesssim \lambda^{-3/8}\|u_0\|_{H^{-3/16}}^2.
			\end{align*}
			To show the second inequality in the former argument we use that the carnality of $\{N\in 2^{\mathbb{N}_0}:N^2\sim \lambda M\}$ is uniformly bounded for $M$.
		\end{proof}
		We give some refinement for the case $\lambda N_1\sim \lambda N_2\sim N^2$ in following two lemmas.
		\begin{lemma}\label{resonrefi}
			Let $N\gg 1$, $0<c\ll 1\ll C$ and $I_1, I_2$ be intervals included in $[-CN^2,-cN^2]\cup [cN^2,CN^2]$ with length less than $cN$. We assume $J:=I_1-I_2\subset [-CN,cN]\cup [cN,CN]$ and $|\xi^3+\lambda\xi_1^2-\lambda\xi_2^2|\gtrsim N$ for any $\xi\in J, \xi_j\in I_j$, $\xi_1+\xi_2 = \xi$. Then we have
			\begin{align*}
				\left\|\eta P_J\mathscr{B} (S_\lambda(t)P_{I_1} u_0 \overline{S_\lambda(t)P_{I_2} u_0})\right\|_{U^2_K}\lesssim \lambda^{-1/2}\|u_0\|_{L^2}^2.
			\end{align*}
		\end{lemma}
		\begin{proof}[\textbf{Proof}]
			Let $L = cN$. Since $|\xi^3+\lambda\xi_1^2-\lambda\xi_2^2|>L/2$ for $\xi_j\in I_j, \xi\in J, \xi_1+\xi_2 = \xi$, we have
			$$\int_{\mathbb{R}^2}\partial_x(S_\lambda(t)P_{I_1} u_0 S_\lambda(t)P_{I_2} u_0)Q^K_{\leq L}P_J \bar{v}~dxdt = 0.$$
			By the high modulation estimate and Lemma 4.2, (4.2) in \cite{banchenzhang} one has
			\begin{align*}
				&\quad\left|\int_{\mathbb{R}^2}\partial_x(S_\lambda(t)P_{I_1} u_0 \overline{S_\lambda(t)P_{I_2} u_0})Q_{>L}^KP_J \bar{v}~dxdt\right|\\
				&\lesssim \|P_J(S_\lambda(s)P_{I_1} u_0 \overline{S_\lambda(s)P_{I_2} u_0})\|_{L^2_{t,x}}\|\partial_xQ^K_{>L}P_J v\|_{L^2_{t,x}}\\
				&\lesssim (\lambda N)^{-1/2}\|u_0\|_{L^2}^2L^{-1/2}N\|v\|_{V^2_K}\sim \lambda^{-1/2}\|u_0\|_{L^2}^2\|v\|_{V^2_K}.
			\end{align*}
			By the duality we conclude the proof.
		\end{proof}
		\begin{lemma}\label{resonrefi3}
			Let $N\gg 1$, $0<c\ll 1\ll C$ and $I_1, I_2$ be intervals included in $[-CN^2,-cN^2]\cup [cN^2,CN^2]$ with length $1/N$. We assume $J$ included in $[-CN,-cN]\cup [cN,CN]$ with length $1/N^2$ and $|\xi^3+\lambda\xi_1^2-\lambda\xi_2^2|\lesssim 1$ for any $\xi\in J, \xi_j\in I_j$, $\xi_1+\xi_2 = \xi$. Then we have
			\begin{align*}
				\left\|\eta P_J\mathscr{B}(S_\lambda(t)P_{I_1} u_0 \overline{S_\lambda(t)P_{I_2} u_0})\right\|_{U^2_K}\lesssim \|u_0\|_{L^2}^2.
			\end{align*}
		\end{lemma}
		\begin{proof}[\textbf{Proof}]
			Let $P_{I_1}u_0 = u_1, P_{I_2}u_0 = u_2$. By the definition of $U^2_K$ and the Taylor expansion one has
			\begin{align*}
				&\quad\left\|\eta P_J\mathscr{B}(S_\lambda(t)P_{I_1} u_0 \overline{S_\lambda(t)P_{I_2} u_0})\right\|_{U^2_K}\\
				&= \left\|\eta P_J\partial_x \int_0^t K(-s)(S_\lambda(s)u_1 \overline{S_\lambda(s)u_2)}~ds\right\|_{U^2_t(L^2_x)}\\
				&\sim \left\|\eta(t) \xi\chi_{\xi \in J} \int_0^t e^{-is\xi^3}\int_{\xi_1+\xi_2 = \xi}e^{-i\lambda s\xi_1^2+i\lambda s \xi_2^2}\widehat{u_1}(\xi_1)\overline{\widehat{u_2}}(-\xi_2)~ds\right\|_{U^2_t(L^2_\xi)}\\
				&\lesssim \sum_{k=0}^\infty\frac{1}{k!} \left\|\eta(t) \xi\chi_{\xi \in J} \int_0^t \int_{\xi_1+\xi_2 = \xi}\widehat{u_1}(\xi_1)\overline{\widehat{u_2}}(-\xi_2)s^k(\xi^3+\lambda\xi_1^2-\lambda\xi_2^2)^k~ds\right\|_{U^2_tL^2_\xi}\\
				&\lesssim \sum_{k=0}^\infty\frac{C^kN}{k!} \left\|\chi_{\xi \in J} \int_{\xi_1+\xi_2 = \xi}|\widehat{u_1}(\xi_1)\widehat{u_2}(-\xi_2)|~ds\right\|_{L^2_\xi}\left\|\eta(t)\frac{t^{k+1}-1}{k+1}\right\|_{U^2_t}\\
				&\lesssim N\|u_1\|_{L^2}\|u_2\|_{L^2}|J|^{1/2}\lesssim \|u_0\|_{L^2}^2.
			\end{align*}
			Here we use $\|\eta(t)t^{k+1}\|_{U^2_t}\lesssim \|\eta(t)t^{k+1}\|_{H^{1}}\lesssim k$.
		\end{proof}
		\begin{lemma}\label{uusmooth}
			$\|N^{1/4}P_N F[u_0]\|_{l^2_{N}L_x^\infty L_t^2}\lesssim \lambda^{-1}\|u_0\|_{H^{-3/16}}^2$.
		\end{lemma}
		\begin{proof}[\textbf{Proof}]
			Firstly by Lemma \ref{estiforF} and the local smoothing estimate one has
			$$\|N^{1/4}P_N (F[u_0]-F[P_{\sim N^2/\lambda}u_0])\|_{l^2_{N}L_x^\infty L_t^2}\lesssim \lambda^{-1}\|u_0\|_{H^{-3/16}}^2$$
			Thus we only need to show
			\begin{align*}
				\|N^{1/4}P_N F[P_{\sim N^2/\lambda}u_0]\|_{l^2_NL^\infty_xL^2_t}\lesssim \lambda^{-1} \|u_0\|_{H^{-3/16}}^2.
			\end{align*}
			If $N\lesssim 1$, by the Sobolev inequality, then
			\begin{align*}
				N^{1/4}\|P_N F[P_{\sim N^2/\lambda}u_0]\|_{L^\infty_xL^2_t}\lesssim \|P_{\sim N^2/\lambda}u_0\|_{L^4_x}^2\lesssim \lambda^{-7/8} \|u_0\|_{H^{-3/16}}^2.
			\end{align*}
			If $N\gg 1$, we claim 
			\begin{equation}\label{claimsmo}
				\left\|P_NF[P_{\sim N^2/\lambda}u_0]\right\|_{L_x^\infty L_t^2}\lesssim \lambda^{-1/2}N^{-1}\|u_0\|_{L^2}^2.
			\end{equation}
			Let $I_j$, $j= 1,2 $ be intervals included in $[-CN^2/\lambda,cN^2/\lambda]\cup[cN^2/\lambda,CN^2/\lambda]$ with length $cN$. Also we assume for any $\xi_j\in I_j$, $j = 1,2$ one has $|\eta_1-\eta_2|\sim N$. We only need to show
			$$\left\|\eta P_N\mathscr{B}(S_\lambda(t)P_{I_1}u_0\overline{S_\lambda(t)P_{I_2} v_0})\right\|_{L_x^\infty L_t^2}\lesssim \lambda^{-1/2}N^{-1}\|u_0\|_{L^2}\|v_0\|_{L^2}.$$
			There exists a interval $J$ included in $[-N, -N/2]\cup [N/2, N]$ with length $1$ such that $|\xi^3+\lambda\xi_1^2-\lambda\xi_2^2|\gtrsim N^2$ for any $\xi\in [-N,-N/2]\cup [N/2,N]\setminus J:=\hat{J}$, $\xi_j\in I_j$, $\xi = \xi_1+\xi_2$. Then by Lemma \ref{resonrefi}, the local smoothing estimate, and the orthogonality we only need to show
			$$\left\|\eta P_{J}\mathscr{B}(S_\lambda(t)P_{I_1}u_0\overline{S_\lambda(t)P_{I_2} v_0})\right\|_{L_x^\infty L_t^2}\lesssim \lambda^{-1/2}N^{-1}\|u_0\|_{L^2}\|v_0\|_{L^2}$$
			where $I_1,I_2, J$ are intervals with length $1$, $|\xi_j|\sim N^2$ for any $\xi_j\in I_j$, $|\xi|\sim N$ for any $\xi\in J$.
			
			There exists a interval $J_1$ included in $J$ with length $1/N$ such that $|\xi^3+\lambda\xi_1^2-\lambda\xi_2^2|\gtrsim N$ for any $\xi\in J\setminus J_1$, $\xi_j\in I_j$, $\xi = \xi_1+\xi_2$. Then again by Lemma \ref{resonrefi}, the local smoothing estimate, and  the orthogonality we only need to show
			$$\left\|\eta P_{J_1}\mathscr{B}(S_\lambda(t)P_{I_1}u_0\overline{S_\lambda(t)P_{I_2} v_0})\right\|_{L_x^\infty L_t^2}\lesssim \lambda^{-1/2}N^{-1} \|u_0\|_{L^2}\|v_0\|_{L^2}$$
			where $I_1,I_2, J_1$ are intervals with length $1/N$.
			
			There exists a interval $J_2$ included in $J_1$ with length $1/N^2$
			such that one has $|\xi^3+\lambda\xi_1^2-\lambda\xi_2^2|\lesssim 1$ for any $\xi\in J_2$, $\xi_j\in I_j$, $\xi = \xi_1+\xi_2$. Then by Lemma \ref{resonrefi3} and  the local smoothing estimate one has
			\begin{align*}
				&\quad\left\|\eta P_{J_2}\mathscr{B}(S_\lambda(t)P_Iu_0\overline{S_\lambda(t)P_J v_0})\right\|_{L_x^\infty L_t^2}\\
				&\lesssim N^{-1}\left\|\eta P_{J_2}\mathscr{B}(S_\lambda(t)P_Iu_0\overline{S_\lambda(t)P_J v_0})\right\|_{U^2_K}\\
				&\lesssim N^{-1}\|u_0\|_{L^2}\|v_0\|_{L^2}.
			\end{align*}
			Note that
			\begin{align*}
				&\quad\mathscr{F}_x\left(\eta(t)P_{J_1\setminus J_2}\mathscr{B}(S_\lambda(t)P_Iu_0\overline{S_\lambda(t)P_J v_0})\right)\\
				&\sim \eta(t)\xi\int_0^t\int_{\xi_1+\xi_2 = \xi} \chi_{|\xi|\sim J_1\setminus J_2, \xi_j\in I_j}e^{i(t-s)\xi^3}e^{-i\lambda s\xi_1^2+i\lambda s\xi_2^2}\widehat{u_0}(\xi_1)\overline{\widehat{v_0}}(-\xi_2)\\
				&\sim \eta(t)\int_{\xi_1+\xi_2 = \xi} \chi_{|\xi|\sim J_1\setminus J_2, \xi_j\in I_j}\frac{e^{it\xi^3}-e^{i\lambda t\xi(\xi_2-\xi_1)}}{\xi^2+\lambda\xi_1-\lambda\xi_2}\widehat{u_0}(\xi_1)\overline{\widehat{v_0}}(-\xi_2)\\
				&:= \Lambda_1+\Lambda_2.
			\end{align*}
			By the local smoothing estimate we have
			\begin{align*}
				\|\mathscr{F}_\xi^{-1}\Lambda_1\|_{L_x^\infty L_t^2}&\lesssim N^{-1}\left\|\int_{\xi_1+\xi_2 = \xi} \chi_{|\xi|\sim J_1\setminus J_2, \xi_j\in I_j}\frac{\widehat{u_0}(\xi_1)\overline{\widehat{v_0}}(-\xi_2)}{\xi^2+\lambda\xi_1-\lambda\xi_2}\right\|_{L^2_\xi}\\
				&\lesssim N^{-1}\|u_0\|_{L^2}\|v_0\|_{L^2}\left\|\frac{\chi_{|\xi|\sim J_1\setminus J_2, \xi_1\in I_1, \xi-\xi_1\in I_2}}{\xi^2+2\lambda\xi_1-\lambda\xi}\right\|_{L^\infty_{\xi_1} L^2_\xi}\\
				&\lesssim N^{-1}\|u_0\|_{L^2}\|v_0\|_{L^2}.
			\end{align*}
			Let $\xi_j^0$ be the center of $I_j$ respectively. By the Taylor expansion we have
			\begin{align*}
				&\quad\|\mathscr{F}^{-1}_\xi \Lambda_2\|_{L^\infty_x L^2_t}\\
				&\lesssim \sum_{k = 0}^\infty \frac{1}{k!}\Bigg\|\eta(t)\mathscr{F}_\xi^{-1}\int_{\xi_1+\xi_2 = \xi} \chi_{|\xi|\sim J_1\setminus J_2, \xi_j\in I_j}\frac{e^{i\lambda t\xi(\xi_2^0-\xi_1^0)}}{\xi^2+\lambda\xi_1-\lambda\xi_2}\widehat{u_0}(\xi_1)\overline{\widehat{v_0}}(-\xi_2)\\
				&\qquad\qquad\cdot (i\lambda t\xi(\xi_2-\xi_2^0-\xi_1+\xi_1^0))^k \Bigg\|_{L_x^\infty L_t^2}\\
				&\lesssim \sum_{k = 0}^\infty \frac{1}{k!}\Bigg\|\mathscr{F}^{-1}_\xi\int_{\xi_1+\xi_2 = \xi} \chi_{|\xi|\sim J_1\setminus J_2, \xi_j\in I_j}\frac{e^{i\lambda t\xi(\xi_2^0-\xi_1^0)}}{\xi^2+\lambda\xi_1-\lambda\xi_2}\widehat{u_0}(\xi_1)\overline{\widehat{v_0}}(-\xi_2)\\
				&\qquad\qquad\cdot \xi^k(\xi_2-\xi_2^0-\xi_1+\xi_1^0)^k\Bigg\|_{L_x^\infty L_t^2}.
			\end{align*}
			Note that $|\xi_2^0-\xi_1^0|\sim N^2$. By the local smoothing estimate one has
			\begin{align*}
				\|\mathscr{F}^{-1}_\xi \Lambda_2\|_{L^\infty_x L^2_t}&\lesssim N^{-1}\sum_{k = 0}^\infty \frac{1}{k!}\Bigg\|\int_{\xi_1+\xi_2 = \xi} \chi_{|\xi|\sim J_1\setminus J_2, \xi_j\in I_j}\frac{\widehat{u_0}(\xi_1)\overline{\widehat{v_0}}(-\xi_2)}{\xi^2+\lambda\xi_1-\lambda\xi_2}\\
				&\qquad\qquad \cdot \xi^k(\xi_2-\xi_2^0-\xi_1+\xi_1^0)^k \Bigg\|_{L_\xi^2}\\
				&\lesssim \sum_{k = 0}^\infty \frac{C^k}{k!}N^{-1}\Bigg\|\int_{\xi_1+\xi_2 = \xi} \chi_{|\xi|\sim J_1\setminus J_2, \xi_j\in I_j}\frac{\widehat{u_0}(\xi_1)\overline{\widehat{v_0}}(-\xi_2)}{\xi^2+\lambda\xi_1-\lambda\xi_2} \Bigg\|_{L_\xi^2}\\
				&\lesssim N^{-1}\|u_0\|_{L^2}\|v_0\|_{L^2}.
			\end{align*}
			We conclude the proof of \eqref{claimsmo}. Then
			\begin{align*}
				\|N^{1/4}P_NF[P_{\sim N^2/\lambda}u_0]\|_{l^2_{N\gg 1}L_x^\infty L_t^2}&\lesssim \lambda^{-1/2}\|N^{1/4}N^{-1}\|P_{\sim N^2/\lambda}u_0\|_{L^2}^2\|_{l^2_N}\\
				&\lesssim \lambda^{-7/8}\|u_0\|_{H^{-3/16}}^2.
			\end{align*}
			Combining the nonresonant and low frequency cases,  we finish the proof.
		\end{proof}
		Recall the definition for the norm $Z$ by \eqref{normz}. We have:
		\begin{coro}\label{Finz}
			Combining Lemmas \ref{estiforF}, \ref{estiforF2}, and \ref{uusmooth} then $$\|F[u_0]\|_{Z}\lesssim \lambda^{-1}\|u_0\|_{H^{-3/16}}^2.$$
		\end{coro}
		
		\section{Bilinear estimates for KdV equation}\label{birefforkd}
		We mainly use the argument in \cite{WangHHG}, \S 5.3.
		\begin{lemma}\label{lowfr}
			Let $N_1, N_2\geq 2$, $N_1\sim N_2$. Then
			\begin{align*}
				\|\eta P_{\lesssim 1}\mathscr{B}(P_{N_1}uP_{N_2}v)\|_{L_x^2L_t^\infty}\lesssim N_1^{-3/2} \|u\|_{V^2_{K}}\|v\|_{V^2_{K}}.
			\end{align*}
		\end{lemma}
		\begin{proof}[\textbf{Proof}]
			If $N_1\lesssim 1$ we have
			\begin{align*}
				\|\eta P_{\lesssim 1}\mathscr{B}(P_{N_1}uP_{N_2}v)\|_{L_x^2L_t^\infty}&\lesssim \|\tilde{\eta}P_{N_1}uP_{N_2}v\|_{L_t^1L_x^2}\\
				&\lesssim \|P_{N_1}u\|_{L_t^\infty L^4_x}\|P_{N_2}u\|_{L^\infty_t L^4_x}\\
				&\lesssim N_1^{-3/2}\|u\|_{L^\infty
				_t L_x^2}\|v\|_{L^\infty_t L^2_x}.
			\end{align*}
			Thus we consider the case $N_1\sim N_2\gg 1$. Let $L_{\min} = \min\{L_1,L_2\}, L_{\max} = \max\{L_1,L_2\}$. We decompose $P_{\lesssim 1}$ into $\sum_{K\lesssim 1}P_K$. By maximal function estimate, Lemma 5.7, (5.37) and (5.38) in \cite{WangHHG} one has
			\begin{align*}
				&\quad\sum_{K\lesssim N_1^{-2/3}}\|\eta P_K(\mathscr{B}(P_{N_1}uP_{N_2}v))\|_{L_x^2L_t^\infty}\\
				&\lesssim \sum_{K\lesssim N_1^{-2/3}}\sum_{L_1,L_2,L}KL^{-1/2}\|Q_{L_1}^KP_{N_1} uQ_{L_2}^KP_{N_2}v\|_{L^2_{t,x}}\\
				&\lesssim \sum_{K\lesssim N_1^{-2/3}}K\sum_{L_1,L_2,L}L^{-1/2}\min\{L_{\min}^{1/2}K^{1/2},L_{\min}^{1/2}L^{1/2}N_1^{-1}\}\\
				&\hspace{80pt}\cdot\|P_{N_1}Q_{L_1}^Ku\|_{L^2_{t,x}}\|P_{N_2}Q_{L_2}^Kv\|_{L_{t,x}^2}\\
				&\lesssim \sum_{K\lesssim N_1^{-2/3}}K\sum_{L_1,L_2,L}(LL_{\max})^{-1/2}\min\{K^{1/2},L^{1/2}N_1^{-1}\}\|u\|_{V^2_K}\|v\|_{V^2_K}\\
				&\lesssim \sum_{K\lesssim N_1^{-2/3}}K\sum_{L}L^{-1/2}\min\{K^{1/2},L^{1/2}N_1^{-1}\}\|u\|_{V^2_K}\|v\|_{V^2_K}\\
				&\lesssim N_1^{-3/2}\|u\|_{V^2_K}\|v\|_{V^2_K}.
			\end{align*}
			If $K\geq N_1^{-2/3}$, we firstly manipulate the case $L_{\max}\geq N_1^{5/4}$. Similar to the former argument one has
			\begin{equation}\label{highmodu}
				\begin{aligned}
					&\quad\sum_{N_1^{-2/3}\leq K\lesssim 1}\sum_{L_{\max}\geq N_1^{5/4}}\|\eta P_K(\mathscr{B}(P_{N_1}Q^K_{L_1}uP_{N_2}Q^K_{L_2}v))\|_{L_x^2L_t^\infty}\\
					&\lesssim \sum_{N_1^{-2/3}\leq K\lesssim 1}\sum_{L_{\max}\geq N_1^{5/4},L}K(LL_{\max})^{-1/2}\min\{K^{1/2},L^{1/2}N_1^{-1}\}\|u\|_{V^2_K}\|v\|_{V^2_K}\\
					&\lesssim\sum_{L_{\max}\geq N_1^{5/4}}L_{\max}^{-1/2}N_1^{-1}\log(N_1)\|u\|_{V^2_K}\|v\|_{V^2_K}\\
					&\lesssim N_1^{-3/2}\|u\|_{V^2_K}\|v\|_{V^2_K}.
				\end{aligned}
			\end{equation}
			Thus we only need to estimate $L_{\max}< N_1^{5/4}$ for $N_1^{-2/3}\leq K\lesssim 1$. Let
			\begin{align*}
				f_{N_1,L_1}(\tau,\xi) &= \mathscr{F}_{t,x}(P_{N_1}Q^K_{L_1}u)(\tau+\xi^3,\xi),\\
				g_{N_2,L_2}(\tau,\xi) & = \mathscr{F}_{t,x}(P_{N_2}Q^K_{L_2}u)(\tau+\xi^3,\xi).
			\end{align*}
			Then
			\begin{align*}
				&\quad\mathscr{F}_{x}\left(\sum_{N_1^{-2/3}\leq K\lesssim 1}\sum_{L_{\max}< N_1^{5/4}}P_K(\eta\mathscr{B}(P_{N_1}Q^K_{L_1}uP_{N_2}Q^K_{L_2}v))\right)\\
				&\sim \sum_{N_1^{-2/3}\leq K\lesssim 1}\sum_{L_{\max}< N_1^{5/4}}\eta(t)\int_0^t\psi(\xi/K)\xi e^{i(t-s)\xi^3}\\
				&\quad\cdot\int_{\mathbb{R}^2_{\tau_1,\tau_2}}\int_{\xi = \xi_1+\xi_2}f_{N_1,L_1}(\tau_1,\xi_1)g_{N_2,L_2}(\tau_2,\xi_2)e^{it(\tau_1+\xi_1^3)+it(\tau_2+\xi_2^3)}\\
				&\sim \sum_{N_1^{-2/3}\leq K\lesssim 1}\sum_{L_{\max}< N_1^{5/4}}\eta(t)\psi(\xi/K)\xi\int_{\mathbb{R}^2_{\tau_1,\tau_2}}\int_{\xi = \xi_1+\xi_2} \\
				&\quad\cdot\frac{e^{it(\tau_1+\tau_2+\xi_1^3+\xi_2^3)}-e^{it\xi^3}}{\tau_1+\tau_2+\xi_1^3+\xi_2^3-\xi^3}f_{N_1,L_1}(\tau_1,\xi_1)g_{N_2,L_2}(\tau_2,\xi_2)\\
				&:=\mathscr{F}_xI-\mathscr{F}_xII.
			\end{align*}
			Note that on the support set of $f_{N_1,L_1}(\tau_1,\xi_1)g_{N_2,L_2}(\tau_2,\xi_2)$ one has $|\tau_1+\tau_2+\xi_1^3+\xi_2^3-\xi^3|\sim N_1^2|\xi|$ since $L_{\max}< N_1^{5/4}$, $|\xi|\sim K\geq N_1^{-2/3}$.
			By the maximal function estimate and the Plancherel identity one has
			\begin{align*}
				&\quad\|II\|_{L_x^2L_t^\infty}\\
				&\sim \left\|\int_{\substack{\xi = \xi_1+\xi_2,\\\tau_1,\tau_2}}\sum_{N_1^{-2/3}\leq K\lesssim 1}\sum_{L_{\max}< N_1^{5/4}}\frac{\psi(\xi/K)\xi f_{N_1,L_1}(\tau_1,\xi_1)g_{N_2,L_2}(\tau_2,\xi_2)}{\tau_1+\tau_2+\xi_1^3+\xi_2^3-\xi^3}\right\|_{L^2_\xi}\\
				&\lesssim N_1^{-2} \int_{\mathbb{R}_{\tau_1,\tau_2}^2}\sum_{L_{\max}< N_1^{5/4}}\| f_{N_1,L_1}(\tau_1,\cdot)\|_{L^2}\|g_{N_2,L_2}(\tau_2,\cdot)\|_{L^2}\\
				&\lesssim N_1^{-2}\sum_{L_{\max}< N_1^{5/4}}L_1^{1/2}L_2^{1/2}\|f_{N_1,L_1}(\tau,\xi)\|_{L^2_{\tau,\xi}}\|f_{N_2,L_2}\|_{L^2_{\tau,\xi}}\\
				&\lesssim N_1^{-2}\sum_{L_{\max}<N_1^{5/4}}\|u\|_{V^2_K}\|v\|_{V^2_K}\lesssim N_1^{-3/2}\|u\|_{V^2_K}\|v\|_{V^2_K}.
			\end{align*}
			Define
			\begin{align*}
				\mathscr{F}_xI'&:= \sum_{N_1^{-2/3}\leq K\lesssim 1}\sum_{L_{\max}< N_1^{5/4}}\eta(t)\psi(\xi/K)\xi\int_{\mathbb{R}^2_{\tau_1,\tau_2}}\int_{\xi = \xi_1+\xi_2} \\
				&\hspace{80pt}\frac{e^{it(\tau_1+\tau_2+\xi_1^3+\xi_2^3)}}{\xi_1^3+\xi_2^3-\xi^3}f_{N_1,L_1}(\tau_1,\xi_1)g_{N_2,L_2}(\tau_2,\xi_2)\\
				&=\sum_{N_1^{-2/3}\leq K\lesssim 1}\sum_{L_{\max}< N_1^{5/4}}\eta(t)\psi(\xi/K)\int_{\xi = \xi_1+\xi_2} \\
				&\hspace{80pt}\frac{1}{-3\xi_1\xi_2}\mathscr{F}_x(P_{N_1}Q^K_{L_1}u)(t,\xi_1)\mathscr{F}_x(P_{N_2}Q^K_{L_2}v)(t,\xi_2).
			\end{align*}
			Let $f:= \sum_{L_1< N_1^{5/4}}f_{N_1,L_2}$, $g: = \sum_{L_2< N_1^{5/4}}g_{N_2,L_2}$.
			Then
			\begin{align*}
				\mathscr{F}_xI-\mathscr{F}_xI'
				&= \sum_{N_1^{-2/3}\leq K\lesssim 1}\sum_{L_{\max}<N_1^{5/4}}\sum_{k=1}^\infty\eta(t)\psi(\xi/K)\xi\int_{\mathbb{R}^2_{\tau_1,\tau_2}}\int_{\xi = \xi_1+\xi_2} \\
				&\quad\frac{(-1)^k(\tau_1+\tau_2)^ke^{it(\tau_1+\tau_2+\xi_1^3+\xi_2^3)}}{(\xi_1^3+\xi_2^3-\xi^3)^{k+1}}f_{N_1,L_1}(\tau_1,\xi_1)g_{N_2,L_2}(\tau_2,\xi_2).
			\end{align*}
			Thus by the Strichartz estimate \cite{kenig1991oscillatory} one has
			\begin{align*}
				&\quad\|I-I'\|_{L_x^2L_t^\infty}\\
				&\lesssim \int_{\mathbb{R}^2_{\tau_1,\tau_2}}\sum_{k=1}^\infty  \sum_{N_1^{-2/3}\leq K\lesssim 1}\sum_{L_{\max}< N_1^{5/4}}|\tau_1+\tau_2|^kK^{-k}  \\
				&\cdot\left\|\eta(t)\mathscr{F}^{-1}_\xi\int_{\xi = \xi_1+\xi_2}\left(\frac{e^{it(\xi_1^3+\xi_2^3)}}{\xi_1^{k+1}\xi_2^{k+1}}f_{N_1,L_1}(\tau_1,\xi_1)g_{N_2,L_2}(\tau_2,\xi_2)\right)\right\|_{L^2_xL_t^\infty}\\
				&\lesssim \int_{\mathbb{R}^2_{\tau_1,\tau_2}}\sum_{k=1}^\infty  \sum_{N_1^{-2/3}\leq K\lesssim 1}\sum_{L_{\max}< N_1^{5/4}}N_1^{5k/4}K^{-k}  \\
				&\cdot\|\mathscr{F}^{-1}_\xi(e^{it\xi^3}f_{N_1,L_1}(\tau_1,\xi)/\xi^{k+1})\|_{L_x^4L_t^\infty}\|\mathscr{F}^{-1}_\xi(e^{it\xi^3}f_{N_2,L_2}(\tau_2,\xi)/\xi^{k+1})\|_{L_x^4L_t^\infty}\\
				&\lesssim \int_{\mathbb{R}^2_{\tau_1,\tau_2}}\sum_{k=1}^\infty  \sum_{L_{\max}< N_1^{5/4}}N_1^{23k/12-2(k+1)+1/2}\|f_{N_1,L_1}(\tau_1,\cdot)\|_{L^2}\|f_{N_2,L_2}(\tau_2,\cdot)\|_{L^2}\\
				&\lesssim \sum_{L_{\max}< N_1^{5/4}}N_1^{-37/12}(L_1L_2)^{1/2}\|Q_{L_1}^KP_{N_1}u\|_{L^2_{t,x}}\|Q_{L_2}^KP_{N_2}v\|_{L^2_{t,x}}\\
				&\lesssim\sum_{L_{\max}< N_1^{5/4}}N_1^{-37/12}\|u\|_{V^2_K}\|v\|_{V^2_K}\lesssim N^{-3/2}\|u\|_{V^2_K}\|v\|_{V^2_K}.
			\end{align*}
			Note that there exists constant $\tilde{c}$ such that
			$$I' = \tilde{c}\eta(t)P_{N^{-2/3}\leq \cdot\lesssim 1}((\partial_x^{-1}P_{N_1}Q^K_{< N_1^{5/4}}u)(\partial_x^{-1}P_{N_1}Q^K_{<N_1^{5/4}}v)).$$
			By the H\"{o}lder inequality we have
			\begin{align*}
				\|I'\|_{L_x^2L_t^\infty}&\lesssim \|\partial_x^{-1}P_{N_1}Q^K_{< N_1^{5/4}}u\|_{L_x^4L_t^\infty}\|\partial_x^{-1}P_{N_1}Q^K_{< N_1^{5/4}}v\|_{L_x^4L_t^\infty}\\
				&\lesssim N^{-2}\|P_{N_1}Q^K_{< N_1^{5/4}}u\|_{L_x^4L_t^\infty}\|P_{N_2}Q^K_{< N_1^{5/4}}v\|_{L_x^4L_t^\infty}\\
				&\lesssim N^{-3/2}\|Q^K_{< N_1^{5/4}}u\|_{U^4_K}\|Q^K_{< N_1^{5/4}}v\|_{U^4_K}\\
				&\lesssim N^{-3/2}\|Q^K_{< N_1^{5/4}}u\|_{V^2_K}\|Q^K_{< N_1^{5/4}}v\|_{V^2_K}\lesssim N^{-3/2}\|u\|_{V^2_K}\|v\|_{V^2_K}.
			\end{align*}
			We conclude the proof.
		\end{proof}
		\begin{prop}\label{uuF}
			$\|\eta\mathscr{B}(uv)\|_Y\lesssim \|u\|_Y\|v\|_Y$.
		\end{prop}
		\begin{proof}[\textbf{Proof}]
			By Lemma \ref{lowfr} we have
			\begin{align*}
				&\quad\|\eta P_1\mathscr{B}(uv)\|_{L_x^2L_t^\infty}\\
				&\leq\|\eta P_1\mathscr{B}(P_{\lesssim 1}uP_{\lesssim 1}v)\|_{L_x^2L_t^\infty}+ \sum_{N_1\sim N_2\gg 1}\|\eta P_1\mathscr{B}(P_{N_1}uP_{N_2}v)\|_{L_x^2L_t^\infty}\\
				&\lesssim \|\tilde{\eta}P_{\lesssim 1}uP_{\lesssim 1}v\|_{L_t^1L_x^2}+\sum_{N_1\sim N_2\gg 1}N_1^{-3/2}\|P_{N_1}u\|_{V^2_K}\|P_{N_2}v\|_{V^2_K}\\
				&\lesssim \|P_{\lesssim 1}u\|_{L_t^\infty L_x^4}\|P_{\lesssim 1}v\|_{L_t^\infty  L_x^4}+\|u\|_Z\|v\|_Z\\
				&\lesssim \|u\|_Z\|v\|_Z.
			\end{align*}
			 Let $N\geq 2$. By the triangle inequality we have
			 \begin{align*}
			 	&\quad \|\eta P_N\mathscr{B}(uv)\|_{U^2_K}\\
			 	&\lesssim \sum_{N_1\sim N}\|\eta P_N\mathscr{B}(P_{N_1}uP_{\lesssim 1}v)\|_{U^2_K}+\sum_{N_2\sim N}\|\eta P_N\mathscr{B}(P_{\lesssim 1}uP_{N_2}v)\|_{U^2_K}\\
			 	&\quad + \sum_{N_1,N_2\gg 1}\|\eta P_N\mathscr{B}(P_{N_1}uP_{N_2}v)\|_{U^2_K}:=I_N+I'_N+II_N.
			 \end{align*}
			 For the term $I$ we use Proposition 5.12 in \cite{WangHHG}.
			 \begin{align*}
			 	\|I_N\|_{U^2_K}&\lesssim \sum_{N_1\sim N}\|\tilde{\eta}^2\partial_x(P_{N_1}uP_{\lesssim 1}v)\|_{L^2_{t,x}}\\
			 	&\lesssim \sum_{N_1\sim N}N\|\tilde{\eta}P_{N_1}u\|_{L_x^\infty L_t^2}\|\tilde{\eta}P_{\lesssim 1}v\|_{L_x^2L_t^\infty}.
			 \end{align*}
			 Then by the local smoothing estimate one has 
			 \begin{align*}
			 	\|N^{-3/4}I_N\|_{l^2_{N\geq 2}U^2_K}&\lesssim \|\tilde{\eta}N^{1/4}P_N u\|_{l^2_{N}L_x^\infty L_t^2}\|\tilde{\eta}P_{\lesssim 1}v\|_{L_x^2L_t^\infty}\\
			 	&\lesssim \|u\|_Z\|v\|_Z
			 \end{align*}
			 By symmetry we control $I'_N$ similarly.
			 
			 For $II_N$ we decompose the summation into four parts.
			 \begin{align*}
			 	II_N & =\sum_{N_1\sim N_2\sim N\gg 1}+\sum_{N_1\sim N\gg N_2\gg 1}+\sum_{N_2\sim N\gg N_1\gg 1}+\sum_{N_1\sim N_2\gg N}\cdots\\
			 	&:= II_N^{(1)}+II_N^{(2)}+II_N^{(3)}+II_N^{(4)}.
			 \end{align*}
			 For $N_1\sim N_2\sim N\gg 1$ we show
			 \begin{align*}
			 	\left|\int_{\mathbb{R}^2} \partial_x(P_{N_1}uP_{N_2}v)P_N \bar{w}~dxdt\right|\lesssim N^{-3/4}\|u\|_{V^2_K}\|v\|_{V^2_K}\|w\|_{V^2_K}
			 \end{align*} 
			 where $u,v,w$ are supported on $[-1,1]\times \mathbb{R}$. Let $L_{\min}, L_{\mathrm{med}}, L_{\max}$ be the maximum, medium and minimum among $L_1,L_2,L$. By Lemma 5.7 in \cite{WangHHG} we have
			 \begin{align*}
			 	&\quad \left|\int_{\mathbb{R}^2} \partial_x(P_{N_1}uP_{N_2}v)P_N \bar{w}~dxdt\right|\\
			 	&\leq \sum_{L_{\max}\gtrsim cN^3} \left|\int_{\mathbb{R}^2} \partial_x(Q_{L_1}^KP_{N_1}uQ^K_{L_2}P_{N_2}v)\overline{P_N Q_{L}^Kw}~dxdt\right|\\
			 	&\lesssim \sum_{L_{\max}\gtrsim cN^3}N L_{\min}^{1/2}L_{\mathrm{med}}^{1/4}N^{-1/4} \|Q_{L_1}^KP_{N_1}u\|_{L_{t,x}^2}\|Q_{L_2}^KP_{N_2}v\|_{L_{t,x}^2}\|Q_{L}^KP_{N}w\|_{L_{t,x}^2}\\
			 	&\lesssim  \sum_{L_{\max}\gtrsim cN^3} L_{\min}^{1/2}L_{\mathrm{med}}^{1/4}N^{3/4} (L_1L_2L)^{-1/2}\|u\|_{V_K^2}\|v\|_{V_K^2}\|w\|_{V_K^2}\\
			 	&\lesssim N^{-3/4}\|u\|_{V_K^2}\|v\|_{V_K^2}\|w\|_{V_K^2}
			 \end{align*}
			 Then by the duality we have
			 \begin{align*}
			 	\|N^{-3/4}II_N^{(1)}\|_{l^2_{N\geq2} U^2_K}&\lesssim \left\|N^{-3/2}\sum_{N_1\sim N_2\sim N\gg 1}\|P_{N_1}u\|_{V^2_K}\|P_{N_2}v\|_{V^2_K}\right\|_{l^2_{N\geq 2}}\\
			 	&\lesssim \|N^{-3/4}P_Nu\|_{l^2_{N\geq 2}V^2_K}\|N^{-3/4}P_Nv\|_{l^2_{N\geq 2}V^2_K}\\
			 	&\lesssim \|u\|_Z\|v\|_Z.
			 \end{align*}
			 For $N_1\sim N\gg N_2\gg 1$ we show
			 \begin{align*}
			 	\left|\int_{\mathbb{R}^2} \partial_x(P_{N_1}uP_{N_2}v)P_N \bar{w}~dxdt\right|\lesssim N^{-1/2}N_2^{1/4}\|u\|_{V^2_K}\|v\|_{V^2_K}\|w\|_{V^2_K}.
			 \end{align*}
			 In fact by Lemma 5.7 (5.38) in \cite{WangHHG} we have
			 \begin{align*}
			 	&\quad\left|\int_{\mathbb{R}^2} \partial_x(P_{N_1}uP_{N_2}v)P_N \bar{w}~dxdt\right|\\
			 	&\leq \sum_{L_1, L_2, L} \left|\int_{\mathbb{R}^2} \partial_x(Q_{L_1}^KP_{N_1}uQ^K_{L_2}P_{N_2}v)\overline{P_N Q_{L}^Kw}~dxdt\right|\\
			 	&\lesssim \sum_{L_1, L_2, L}N(L_1L_2L)^{1/3} N^{-1/2}(N^2N_2)^{-1/2}\\
			 	&\quad\cdot \|Q_{L_1}^KP_{N_1}u\|_{L_{t,x}^2}\|Q_{L_2}^KP_{N_2}v\|_{L_{t,x}^2}\|Q_{L}^KP_{N}w\|_{L_{t,x}^2}\\
			 	&\lesssim  \sum_{L_1, L_2, L}(L_1L_2L)^{1/3} N^{-1/2}N_2^{-1/2} (L_1L_2L)^{-1/2}\|u\|_{V_K^2}\|v\|_{V_K^2}\|w\|_{V_K^2}\\
			 	&\lesssim N^{-1/2}N_2^{-1/2}\|u\|_{V_K^2}\|v\|_{V_K^2}\|w\|_{V_K^2}.
			 \end{align*}
			 Then by  the duality we have
			 \begin{align*}
			 	\|N^{-3/4}II_N^{(2)}\|_{l^2_{N\geq2} U^2_K}&\lesssim \sum_{N_1\sim N\gg N_2\gg 1}N^{-5/4}N_2^{-1/2}\|P_{N_1}u\|_{V^2_K}\|P_{N_2}v\|_{V^2_K}\\
			 	&\lesssim \|N^{-3/4}P_Nu\|_{l^2_{N\geq 2}V^2_K}\|N^{-3/4}P_Nv\|_{l^2_{N\geq 2}V^2_K}\\
			 	&\lesssim \|u\|_Z\|v\|_Z.
			 \end{align*}
			 By the symmetry we can control $II_N^{(3)}$ similarly.
			 
			 Now, we estimate the term $II_N^{(4)}$. Let $N_1\sim N_2\gg N\geq 2$, $L = cN_1^2N$. Firstly we estimate
			 \begin{align*}
			 	\|\eta P_N\mathscr{B}(Q_{\leq L}^K(P_{N_1}u P_{N_2}v))\|_{U^2_K}.
			 \end{align*}
			 Typically we control 
			 \begin{align*}
			 	\int_{\mathbb{R}^2} \partial_x(Q_{>L}^KP_{N_1}uP_{N_2}v)\overline{Q_{\leq L}^KP_N w}~dxdt.
			 \end{align*}
			 By high modulation and transversal estimates, we have
			 \begin{align*}
			 	&\quad \left|\int_{\mathbb{R}^2} \partial_x(Q_{>L}^KP_{N_1}uP_{N_2}v)\overline{Q_{\leq L}P_N w}~dxdt\right|\\
			 	&\lesssim \|Q_{>L}^KP_{N_1}u\|_{L_{t,x}^2}\|P_{N_2}v\partial_xQ^K_{\leq L}P_N w\|_{L_{t,x}^2}\\
			 	&\lesssim L^{-1/2}\|u\|_{V^2_K}N_2^{-1}N\|v\|_{U^2_K}\|w\|_{U^2_K}\\
			 	&\sim N_1^{-2}N^{1/2}\|u\|_{V^2_K}\|v\|_{U^2_K}\|w\|_{U^2_K}.
			 \end{align*}
			 On the other hand by the Strichartz estimate we have
			 \begin{align*}
			 	&\quad \left|\int_{\mathbb{R}^2} \partial_x(Q^K_{>L}P_{N_1}uP_{N_2}v)\overline{Q^K_{\leq L}P_N w}~dxdt\right|\\
			 	&\lesssim \|Q^K_{>L}P_{N_1}u\|_{L_{t,x}^4}\|P_{N_2}v\|_{L_{t,x}^4}\|\partial_xQ^K_{\leq L}P_N w\|_{L_{t,x}^2}\\
			 	&\lesssim \|Q_{>L}^KP_{N_1}u\|_{L_{t,x}^2}^{1/4}\|Q_{>L}^KP_{N_1}u\|_{L_{t,x}^6}^{3/4}\|P_{N_2}v\|_{L_{t,x}^2}^{1/4}\|P_{N_2}v\|_{L_{t,x}^6}^{3/4}N\|w\|_{U^6_K}\\
			 	&\sim N_1^{-1/4}N\|u\|_{V^2_K}\|v\|_{U^6_K}\|w\|_{U^6_K}.
			 \end{align*}
			 Then, by the interpolation \cite{hadac2009well}, one has
			 \begin{align*}
			 	&\quad\left|\int_{\mathbb{R}^2} \partial_x(Q_{>L}^KP_{N_1}uP_{N_2}v)\overline{Q^K
			 		_{\leq L}P_N w}~dxdt\right|\\
			 	&\lesssim N_1^{-2}N^{1/2}(\log N_1)^2 \|u\|_{V^2_K}\|v\|_{V^2_K}\|w\|_{V^2_K}.
			 \end{align*}
			 Thus,
			 \begin{align*}
			 	&\quad \sum_{N_1\sim N_2\gg N\geq 2}N^{-3/4}\|\eta P_N\mathscr{B}(Q^K_{\leq L}(P_{N_1}u P_{N_2}v))\|_{U^2_K}\\
			 	&\lesssim \sum_{N_1\sim N_2\gg N\geq 2}N^{-3/4}N_1^{-2}N^{1/2}(\log N_1)^2 \|P_{N_1}u\|_{V^2_K}\|P_{N_2}v\|_{V^2_K}\\
			 	&\lesssim \|u\|_Z\|v\|_Z.
			 \end{align*}
			 Secondly we estimate $\|\eta P_N\mathscr{B}(Q^K_{\leq L}(P_{N_1}u P_{N_2}v))\|_{U^2_K}$. By high modulation and transversal estimates, we have
			 \begin{align*}
			 	&\quad \left|\int_{\mathbb{R}^2} \partial_x(P_{N_1}uP_{N_2}v)\overline{Q^K_{>L}P_N w}~dxdt\right|\\
			 	&\lesssim \|P_N(P_{N_1}uP_{N_2}v)\|_{L_{t,x}^2}\|\partial_xQ^K_{>L}P_N w\|_{L_{t,x}^2}\\
			 	&\lesssim (NN_1)^{-1/2}\|u\|_{U^2_K}\|v\|_{U^2_K}L^{-1/2}N\|w\|_{V^2_K}\\
			 	&\sim N_1^{-3/2}\|u\|_{U^2_K}\|v\|_{U^2_K}\|w\|_{V^2_K}.
			 \end{align*}
			 Then by the duality one has
			 \begin{align*}
			 	&\quad\sum_{N_1\sim N_2\gg N\geq 2}\|N^{-3/4}\eta P_N\mathscr{B}(Q^K_{\leq L}(P_{N_1}u P_{N_2}v))\|_{U^2_K}\\
			 	&\lesssim \sum_{N_1\sim N_2\gg N\geq 2} N^{-3/4}N_1^{-3/2}\|P_{N_1}u\|_{U^2_K}\|P_{N_2}v\|_{U^2_K}\\
			 	&\lesssim \|u\|_Y\|v\|_Y.
			 \end{align*}
			 We finish the proof of this proposition.
		\end{proof}
		In the proof of Proposition \ref{uuF}, we control many parts by $\|u\|_Z\|v\|_Z$ except $II_N^{(4)}$. Fortunately, we have the following lemma which can help us to control this case when $v = F[u_0]$.		
		\begin{lemma}\label{thekey}
			$2\leq N\ll N_1\sim N_2$.
			\begin{align*}
				&\|\eta P_N\mathscr{B}(P_{N_1}v P_{N_2}F[P_{\sim N_2^2/\lambda}u_0])\|_{U^2_K}\lesssim (\lambda N_1^3/N)^{-1/2} \|v\|_{U^2_K}\|u_0\|_{L^2}^2,\\
				&\|\eta P_N\mathscr{B}(P_{N_1}(F[P_{\sim N_1^2/\lambda}u_0]) P_{N_2}(F[P_{\sim N_2^2/\lambda}u_0]))\|_{U^2_K}\lesssim (\lambda^2N_1^3/N)^{-1/2} \|u_0\|_{L^2}^4.
			\end{align*}
		\end{lemma}
		\begin{proof}[\textbf{Proof}]
			By the duality, we only need to show
			\begin{align*}
				&\quad\left|\int_{\mathbb{R}^2} P_{N_1}v P_{N_2}(F[P_{\sim N_2^2/\lambda}u_0])\partial_xP_Nw~dxdt\right|\\
				&\lesssim (\lambda N_1^3/N)^{-1/2}\|v\|_{U^2_K}\|u_0\|_{L^2}^2\|w\|_{V^2_K}
			\end{align*}
			and
			\begin{align*}
				&\quad\left|\int_{\mathbb{R}^2} P_{N_1}(F[P_{\sim N_1^2/\lambda}u_0]) P_{N_2}(F[P_{\sim N_2^2/\lambda}u_0])\partial_xP_Nw~dxdt\right|\\
				&\lesssim (\lambda^2N_1^3/N)^{-1/2}\|u_0\|_{L^2}^4\|w\|_{V^2_K}
			\end{align*}
			where $v, w$ are supported on $[-1,1]\times \mathbb{R}$. Let $L = cNN_1N_2$. Firstly, we control
			\begin{align*}
				\left|\int_{\mathbb{R}^2} P_{N_1}v P_{N_2}(F[P_{\sim N_2^2/\lambda}u_0])Q_{\leq L}^K\partial_xP_Nw~dxdt\right|.
			\end{align*}
			Since for $|\xi_1|\sim N_1$, $|\xi_1|\sim N_2$, $|\xi|\sim N$, $\xi_1+\xi_2 = \xi$, one has $|\xi_1^3+\xi_2^3-(\xi_1+\xi_2)^3|\gtrsim L$. Typically, we need to control
			$$\left|\int_{\mathbb{R}^2} P_{N_1}v Q^K_{>L}P_{N_2}(F[P_{\sim N_2^2/\lambda}u_0])Q_{\leq L}^K\partial_xP_Nw~dxdt\right|.$$
			By transversal and high modulation estimates we can control this term by
			\begin{align*}
				&\quad\|P_{N_1}vQ_{\leq L}^K\partial_xP_Nw\|_{L^2_{t,x}}\|Q^K_{>L}P_{N_2}(F[P_{\sim N_2^2/\lambda}u_0])\|_{L^2_{t,x}}\\
				&\lesssim N_1^{-1}N\|v\|_{U^2_K}\|w\|_{U^2_K}L^{-1/2}\|P_{N_2}(F[P_{\sim N_2^2/\lambda}u_0])\|_{V^2_K}\\
				&\sim N_1^{-2}N^{1/2}\|v\|_{U^2_K}\|w\|_{U^2_K}\|P_{N_2}(F[P_{\sim N_2^2/\lambda}u_0])\|_{V^2_K}.
			\end{align*}
			Recall that $v$ is supported on $[-1,1]\times \mathbb{R}$. Thus for any $p\geq 1$, one has $\|v\|_{L_{t,x}^2}\lesssim \|v\|_{L_t^\infty L_x^2}\lesssim \|v\|_{U^p}$.
			By the H\"{o}lder inequality and Strichartz estimates, 
			we can also control this term by
			\begin{align*}
				&\quad\|P_{N_1}v\|_{L_{t,x}^3}\|Q_{\leq L}^K\partial_x P_Nw\|_{L_{t,x}^6}\|Q_{>L}^KP_{N_2}(F[P_{\sim N_2^2}u_0])\|_{L^2_{t,x}}\\
				&\lesssim \|P_{N_1}v\|_{L^2_{t,x}}^{1/2}\|P_{N_1}v\|_{L^6_{t,x}}^{1/2}N\|P_Nw\|_{U^6_K}L^{-1/2}\|P_{N_2}(F[P_{\sim N_2^2}u_0])\|_{V^2_K}\\
				&\sim N_1^{-13/12}N^{1/2}\|v\|_{U^6_K}\|w\|_{U^6_K}\|P_{N_2}(F[P_{\sim N_2^2}u_0])\|_{V^2_K}.
			\end{align*}
			Then, by the interpolation \cite{hadac2009well} and Lemma \ref{estiforF2}, one has
			\begin{align*}
				&\quad\left|\int_{\mathbb{R}^2} P_{N_1}v P_{N_2}(F[P_{\sim N_2^2/\lambda}u_0])Q_{\leq L}^K\partial_xP_Nw~dxdt\right|\\
				&\lesssim N_1^{-2}N^{1/2}(\log N_1)^2 \|v\|_{V^2_K}\|w\|_{V^2_K}\|P_{N_2}(F[P_{\sim N_2^2}u_0])\|_{V^2_K}
				\\
				&\lesssim \lambda^{-3/8}N_1^{-2}N^{1/2}(\log N_1)^2 \|v\|_{V^2_K}\|w\|_{V^2_K}\|P_{N_2}(F[P_{\sim N_2^2}u_0])\|_{V^2_K}.
			\end{align*}
			Similarly, one has
			\begin{align*}
				&\quad\left|\int_{\mathbb{R}^2} P_{N_1}(F[P_{\sim N_1^2/\lambda}u_0]) P_{N_2}(F[P_{\sim N_2^2/\lambda}u_0])Q_{\leq L}^K\partial_xP_Nw~dxdt\right|\\
				&\lesssim N_1^{-2}N^{1/2}(\log N_1)^2\|P_{N_1}(F[P_{\sim N_1^2/\lambda}u_0])\|_{V^2_K}\|P_{N_2}(F[P_{\sim N_1^2/\lambda}u_0])\|_{V^2_K}\|w\|_{V^2_K}\\
				&\lesssim \lambda^{-3/4}N_1^{-2}N^{1/2}(\log N_1)^2 \|u_0\|_{L^2}^4\|w\|_{V^2_K}.
			\end{align*}
			For the other part, by the H\"{o}lder inequality,
			\begin{align*}
				&\quad\left|\int_{\mathbb{R}^2} P_{N_1}v P_{N_2}(F[P_{\sim N_2^2/\lambda}u_0])Q_{> L}^K\partial_xP_Nw~dxdt\right|\\
				&\lesssim \|P_{N}(P_{N_1}vP_{N_2}(F[P_{\sim N_2^2/\lambda}u_0]))\|_{L^2_{t,x}}\|Q_{> L}^K\partial_xP_Nw\|_{L^2_{t,x}}\\
				&\lesssim NL^{-1/2}\|P_{N}(P_{N_1}vP_{N_2}(F[P_{\sim N_2^2/\lambda}u_0]))\|_{L^2_{t,x}}\|w\|_{V^2_K}.
			\end{align*}
			To conclude the proof of this lemma, we show
			\begin{align*}
				&\|P_{N}(P_{N_1}vP_{N_2}(F[P_{\sim N_2^2/\lambda}u_0]))\|_{L^2_{t,x}}\lesssim (\lambda NN_1)^{-1/2}\|v\|_{U^2_K}\|u_0\|_{L^2}^2,\\
				&\|P_N(P_{N_1}(F[P_{\sim N_1^2/\lambda}u_0])P_{N_2}(F[P_{\sim N_2^2/\lambda}u_0]))\|_{L^2_{t,x}}\lesssim (NN_1)^{-1/2}\|u_0\|_{L^2}^4.
			\end{align*}
			To show the first one, we only need to show
			$$\|P_{N}(P_{N_1}K(t)v_0P_{N_2}(F[P_{\sim N_2^2/\lambda}u_0]))\|_{L^2_{t,x}}\lesssim (\lambda NN_1)^{-1/2}\|v_0\|_{L^2}\|u_0\|_{L^2}^2.$$
			Let $I_1, I_2$ be intervals included in $[-CN_2^2/\lambda,-cN_2^2/\lambda]\cup [cN_2^2/\lambda,CN_2^2/\lambda]$ with length $cN_2$. Also, we can assume that for any $\eta_1\in I_1$, $\eta_2\in I_2$ one has $|\eta_1-\eta_2|\sim N_2$. It is equivalent to
			\begin{align*}
				&\quad\left\|P_N\left(P_{N_1}K(t)v_0\tilde{\eta}P_{N_2}\mathscr{B}(S_\lambda(t)P_{I_1} u_0 \overline{S_\lambda(t)P_{I_2} u_0})\right)\right\|_{L^2_{t,x}}\\
				&\lesssim (\lambda NN_1)^{-1/2}\|v_0\|_{L^2}\|u_0\|_{L^2}^2.
			\end{align*}
			There exists a interval $J$ included in $[-CN_2,-cN_2]\cup [cN_2,CN_2]$ with length $1$ such that one has $|\xi_2^3+\lambda\eta_1^2-\lambda \eta_2^2|\gtrsim N_2^2$ for any $\xi_2\in [-CN_2,-cN_2]\cup [cN_2,CN_2]\setminus J:= \hat{J}$, $\eta_1\in I_1, \eta_2\in I_2,\eta_1+\eta_2 = \xi_2$. Then by Lemma \ref{resonrefi} and
			\begin{equation}\label{tranvv}
				\|P_N(P_{N_1}K(t)v_0P_{N_2}K(t)v_1)\|_{L^2_{t,x}}\lesssim (NN_1)^{-1/2}\|v_0\|_{L^2}\|v_1\|_{L^2},
			\end{equation}
			we can control this part. Thus we only need to show
			\begin{align*}
				&\quad\left\|P_N\left(P_{N_1}K(t)v_0\tilde{\eta}P_{J}\mathscr{B}(S_\lambda(t)P_{I_1} u_0 \overline{S_\lambda(t)P_{I_2} u_0})\right)\right\|_{L^2_{t,x}}\\
				&\lesssim (\lambda NN_1)^{-1/2}\|v_0\|_{L^2}\|u_0\|_{L^2}^2.
			\end{align*}
			By the orthogonality, we can assume that $I_1, I_2$ are intervals with length $1$.
			
			There exists a interval $J_1$ included in $J$ with length at most $1/N_2$ such that one has $|\xi_2^3+\lambda\eta_1^2-\lambda \eta_2^2|\gtrsim N_2$ for any $\eta\in J\setminus J_1$, $\eta_1\in I_1, \eta_2\in I_2,\eta_1+\eta_2 = \eta$. By Lemma \ref{resonrefi} and the former argument we only need to show
			\begin{align*}
				&\quad\left\|P_N\left(P_{N_1}K(t)v_0\tilde{\eta}P_{J_1}\mathscr{B}(S_\lambda(t)P_{I_1} u_0 \overline{S_\lambda(t)P_{I_2} u_0})\right)\right\|_{L^2_{t,x}}\\
				&\lesssim (\lambda NN_1)^{-1/2}\|v_0\|_{L^2}\|u_0\|_{L^2}^2
			\end{align*}
			where $I_1, I_2$ are intervals with length $1/N_2$.
			
			There exists a interval $J_2$ included in $J_1$ with length $1/N_2^2$ such that one has $|\xi_2^3+\lambda\eta_1^2-\lambda \eta_2^2|\lesssim 1$ for any $\xi_2\in J_2$, $\eta_1\in I_1, \eta_2\in I_2,\eta_1+\eta_2 = \xi_2$. By Lemma \ref{resonrefi3} and following the former argument we only need to show
			\begin{align*}
				&\quad\left\|P_N\left(P_{N_1}K(t)v_0\tilde{\eta}P_{J_1\setminus J_2}\mathscr{B}(S_\lambda(t)P_{I_1} u_0 \overline{S_\lambda(t)P_{I_2} u_0})\right)\right\|_{L^2_{t,x}}\\
				&\lesssim (\lambda NN_1)^{-1/2}\|v_0\|_{L^2}\|u_0\|_{L^2}^2
			\end{align*}
			where $I_1, I_2$ are intervals with length $1/N_2$. Note that
			\begin{align*}
				&\quad \mathscr{F}_x\left(P_N\left(P_{N_1}K(t)v_0\tilde{\eta}P_{J_1\setminus J_2}\mathscr{B}(S_\lambda(t)P_{I_1} u_0 \overline{S_\lambda(t)P_{I_2} u_0})\right)\right)\\
				&\sim \int_{\substack{\xi
						_1+\xi_2 = \xi,\\\eta_1+\eta_2 =\xi_2}}M(\xi,\xi_1,\eta_1)\widehat{v_0}(\xi_1)\widehat{u_0}(\eta_1)\overline{\widehat{u_0}}(-\eta_2)\tilde{\eta}(e^{it(\xi_1^3+\xi_2^3)}-e^{it(\xi_1^3+\lambda\xi_2(\eta_2-\eta_1))})\\
				&:= \Lambda_1+\Lambda_2.
			\end{align*}
			where $M(\xi,\xi_1,\eta_1) = \chi_{|\xi|\sim N,|\xi_1|\sim N_1, \xi_2\in J_1\setminus J_2,\eta_1\in I_1,\eta_2\in I_2}/(\xi_2^2+\lambda\eta_1-\lambda\eta_2)$. Note that $M(\xi,\xi_2,\eta_1)\lesssim N_2$. For $\Lambda_1$, by \eqref{tranvv} and the Plancherel identity we have
			\begin{align*}
				\|\Lambda_1\|_{L_t^2L_\xi^2}&\lesssim (NN_1)^{-1/2}\|v_0\|_{L^2}\\
				&\quad\cdot\left\|\int_{\eta_1+\eta_2 = \xi_2}\frac{\widehat{u_0}(\eta_1)\overline{\widehat{u_0}}(-\eta_2)}{\xi_2^2+\lambda\eta_1-\lambda\eta_2}\chi_{\xi_2\in J_1\setminus J_2,\eta_1\in I_1,\eta_2\in I_2}\right\|_{L^2_{\xi_2}}\\
				&\lesssim (NN_1)^{-1/2}\|v_0\|_{L^2} \|u_0\|_{L^2}^2\left\|\frac{\chi_{\xi_2\in J_1\setminus J_2,\eta_1\in I_1,\eta_2\in I_2}}{\xi_2^2+2\lambda\eta_1-\lambda\xi_2}\right\|_{L^\infty_{\eta_1}L^2_{\xi_2}}\\
				&\lesssim (NN_1)^{-1/2}\|v_0\|_{L^2} \|u_0\|_{L^2}^2.
			\end{align*}
			Let $\eta_j^0$ be the center of $I_j$, $j = 1,2$. Note that
			\begin{align*}
				\partial_{\xi_1}(\xi_1^3+\lambda(\xi-\xi_1)(\eta_2^0-\eta_1^0)) &= 3\xi_1^2-\lambda(\eta_2^0-\eta_1^0)\\
				&= 3\xi_1^2-\xi_2^2+\xi_2^2-\lambda(\eta_2^0-\eta_1^0)\\
				& = 2\xi_1^2+\xi(\xi_1-\xi_2)+\xi_2^2-\lambda(\eta_2^0-\eta_1^0)\sim N_1^2.
			\end{align*}
			Thus by the Taylor expansion and the transversal estimate one has
			\begin{align*}
				&\quad\|\Lambda_2\|_{L_t^2L_\xi^2}\\
				&\lesssim \sum_{k=0}^\infty \Bigg\|\int_{\substack{\xi_1+\xi_2 = \xi,\\ \eta_1+\eta_2 = \xi_2}}M(\xi,\xi_1,\eta_1)\widehat{v_0}(\xi_1)\widehat{u_0}(\eta_1)\overline{\widehat{u_0}}(-\eta_2)e^{it(\xi_1^3+\lambda\xi_2(\eta_2^0-\eta_1^0))}\\
				&\qquad\qquad\cdot\tilde{\eta}\frac{(i\lambda t\xi_2(\eta_2-\eta_2^0-\eta_1+\eta_1^0))^k}{k!}\Bigg\|_{L_t^2L_\xi^2}\\
				&\lesssim \sum_{k=0}^\infty\frac{1}{k!} \Bigg\|\int_{\substack{\xi_1+\xi_2 = \xi,\\ \eta_1+\eta_2 = \xi_2}}M(\xi,\xi_1,\eta_1)\widehat{v_0}(\xi_1)\widehat{u_0}(\eta_1)\overline{\widehat{u_0}}(-\eta_2)e^{it(\xi_1^3+\lambda\xi_2(\eta_2^0-\eta_1^0))}\\
				&\qquad\qquad\cdot(\xi_2(\eta_2-\eta_2^0-\eta_1+\eta_1^0))^k\Bigg\|_{L_t^2L_\xi^2}\\
				&\lesssim \sum_{k=0}^\infty\frac{1}{k!} N_1^{-1}\|v_0\|_{L^2}\Bigg\|\int_{\eta_1+\eta_2 = \xi_2} \frac{\widehat{u_0}(\eta_1)\overline{\widehat{u_0}}(-\eta_2)}{\xi_2^2+\lambda\eta_1-\lambda\eta_2}\\
				&\qquad\qquad\cdot\chi_{\xi_2\in J_1\setminus J_2,\eta_1\in I_1,\eta_2\in I_2}(\xi_2(\eta_2-\eta_2^0-\eta_1+\eta_1^0))^k\Bigg\|_{L_{\xi_2}^2}\\
				&\lesssim \sum_{k=0}^\infty\frac{C^k}{k!} N_1^{-1}\|v_0\|_{L^2}\Bigg\|\int_{\eta_1+\eta_2 = \xi_2} \frac{|\widehat{u_0}(\eta_1)\widehat{u_0}(-\eta_2)|}{\xi_2^2+\lambda\eta_1-\lambda\eta_2}\chi_{\xi_2\in J_1\setminus J_2,\eta_1\in I_1,\eta_2\in I_2}\Bigg\|_{L_{\xi_2}^2}\\
				&\lesssim N_1^{-1}\|v_0\|_{L^2}\|u_0\|_{L^2}^2.
			\end{align*}
			Thus we conclude the proof of the first inequality.
			
			For the second one, similar to the former argument, we reduce it to
			\begin{align*}
				&\Bigg\|\int_{\zeta_1+\zeta_2 +\eta_1+\eta_2 = \xi}\frac{\tilde{\eta}\chi_{|\xi|\sim N}M_1(\zeta_1,\zeta_2)M_2(\eta_1,\eta_2)\widehat{u_0}(\zeta_1)\widehat{u_0}(\zeta_2)\widehat{u_0}(\eta_1)\widehat{u_0}(\eta_2)}{((\zeta_1+\zeta_2)^2+\lambda\zeta_1-\lambda\zeta_2)((\eta_1+\eta_2)^2+\lambda\eta_1-\lambda\eta_2)}\\
				&\qquad \cdot (e^{it(\zeta_1+\zeta_2)^2}-e^{i\lambda t(\zeta_1+\zeta_2)(\zeta_2-\zeta_1)})(e^{it(\eta_1+\eta_2)^2}-e^{i\lambda t(\eta_1+\eta_2)(\eta_2-\eta_1)})\Bigg\|_{L_t^2L_\xi^2}\\
				&\lesssim (NN_1)^{-1/2}\|u_0\|_{L^2}^4
			\end{align*}
			where $M_1(\zeta_1,\zeta_2) = \chi_{ \zeta_1+\zeta_2\in J,\zeta_1\in I_1,\zeta_2\in I_2}, M_2(\eta_1,\eta_2) = \chi_{ \eta_1+\eta_2\in \tilde{J},\eta_1\in \tilde{I}_1,\eta_2\in \tilde{I}_2}$, $I_1, I_2, \tilde{I}_1, \tilde{I}_2$ are intervals included in $[-CN_1^2,-cN_1^2]\cup[cN_1^2,CN_1^2]$, $J, \tilde{J}$ are intervals included in $[-CN_1,cN_1]\cup  [cN_1,CN_1]$ and
			\begin{equation}\label{midresonant}
				N_1^{-1}\leq |(\zeta_1+\zeta_2)^2+\lambda\zeta_1-\lambda\zeta_2|, |(\eta_1+\eta_2)^2+\lambda\eta_1-\lambda\eta_2|\ll N_1.
			\end{equation}
			Recall the former argument. Let $\zeta_1+\zeta_2 = \zeta,\eta_1+\eta_2 = \eta$, we only need to show
			\begin{align*}
				&\Bigg\|\int_{\zeta +\eta = \xi}\frac{\tilde{\eta}\chi_{|\xi|\sim N}M_1(\zeta_1,\zeta_2)M_2(\eta_1,\eta_2)\widehat{u_0}(\zeta_1)\widehat{u_0}(\zeta_2)\widehat{u_0}(\eta_1)\widehat{u_0}(\eta_2)}{(\zeta^2+\lambda\zeta_1-\lambda\zeta_2)(\eta^2+\lambda\eta_1-\lambda\eta_2)}\\
				&\qquad \cdot e^{i\lambda t(\zeta(\zeta_2-\zeta_1)+\eta(\eta_2-\eta_1))}\Bigg\|_{L_t^2L_\xi^2}\\
				&\lesssim (NN_1)^{-1/2}\|u_0\|_{L^2}^4
			\end{align*}
			Let $\zeta_j^0, \eta_j^0$ be the center of $I_j,\tilde{I}_j$ respectively. Note that
			\begin{align*}
				&\quad\partial_{\zeta}(\lambda \zeta(\zeta_2^0-\zeta_1^0)+\lambda(\xi-\zeta)(\eta_2^0-\eta_1^0))\\
				& = \lambda\zeta_2^0-\lambda\zeta_1^0-\lambda(\eta_2^0-\eta_1^0)\\
				& = \zeta^2-\eta^2+\lambda\zeta_2^0-\lambda\zeta_1^0-\zeta^2+\eta^2-\lambda (\eta_2^0-\eta_1^0)\\
				& = \xi(\zeta-\eta)-(\zeta^2+\lambda\zeta_1^0-\lambda\zeta_2^0)+(\eta^2-\lambda \eta_2^0+\eta_1^0).
			\end{align*}
			By \eqref{midresonant}, $|\xi|\sim N, |\zeta-\eta|\sim N_1$, we obtain
			$$|\partial_{\zeta}(\lambda \zeta(\zeta_2^0-\zeta_1^0)+\lambda(\xi-\zeta)(\eta_2^0-\eta_1^0))|\sim NN_1.$$
			Following the same argument for the first inequality, by the Taylor expansion and the transversal estimate, we conclude the proof.
		\end{proof}
		\begin{prop}\label{uftow}
			$\|\eta\mathscr{B}(uF[u_0])\|_{Y}\lesssim \lambda^{-1}\|u\|_{Y}\|u_0\|_{H^{-3/16}}^2$.
		\end{prop}
		\begin{proof}[\textbf{Proof}]
			By the proof of Proposition \ref{uuF} and Lemma \ref{Finz}, we only need to control the following term 
			$$\sum_{N_1\sim N_2\gg N\geq 2}\|\eta P_N\mathscr{B}(P_{N_1}uP_{N_2}F[P_{\sim N_2^2/\lambda}u_0])\|_{U^2_K}.$$
			By Lemma \ref{thekey}, we can control it by
			\begin{align*}
				&\quad \sum_{N_1\sim N_2\gg N\geq 2} N^{-3/4}N_1^{-3/2}N^{1/2}\lambda^{-1/2}\|P_{N_1}u\|_{U^2_K}\|P_{\sim N_2^2/\lambda}u_0\|_{L^2}^2\\
				&\lesssim \lambda^{-7/8}\|u\|_Y\|u_0\|_{H^{-3/16}}^2.
			\end{align*}
			Combining Corollary \ref{Finz} and  Proposition \ref{uuF}, we conclude the proof.
		\end{proof}
		\begin{prop}\label{u04tov}
			$\|\eta\mathscr{B}(F[u_0]^2)\|_{Y}\lesssim \lambda^{-2}\|u_0\|_{H^{-3/16}}^4$.
		\end{prop}
		\begin{proof}[\textbf{Proof}]
			Similar to the argument for Proposition \ref{uftow}, we only need to control the term
			$$\sum_{N_1\sim N_2\gg N\geq 2}N^{-3/4}\|\eta P_N\mathscr{B}(P_{N_1}(F[P_{\sim N_1^2/\lambda}u_0]) P_{N_2}(F[P_{\sim N_2^2/\lambda}u_0]))\|_{U^2_K}.$$
			By Lemma \ref{thekey} we can control it by
			\begin{align*}
				\sum_{N_1\sim N_2\gg N\geq 2} N^{-3/4}N_1^{-3/2}N^{1/2}\lambda^{-1}\|P_{\sim N_2^2/\lambda}u_0\|_{L^2}^4 \lesssim\lambda^{-7/4}\|u_0\|_{H^{-3/16}}^4.
			\end{align*}
			By Corollary \ref{Finz} and Proposition \ref{uuF}, we conclude the proof.
		\end{proof}
		
		Combining Lemmas \ref{linear}--\ref{uvwtov}, Corollary \ref{Finz}, Propositions \ref{uuF}, \ref{uftow}, and \ref{u04tov}, one can construct the solution $(u,w)$ of  \eqref{themodelmanu}. Let $v = w+F[u_0]$. By rescaling we obtain the local solution $(u,v)\in C([0,T];H^{-3/16}\times H^{-3/4})$ of initial system \eqref{model} with $(s_1,s_2) = (-3/16,-3/4)$. For general $(s_1,s_2)$, see the argument in \cite{banchenzhang}. One can use the method in this paper to manipulate the case $s_1,s_2<0$ with some modification. However there exists easier argument. For example if $s_2>-3/4$, we do not need rescaling. One can use the norms
		$$\|u\|_{X_\lambda^{s_1}}:= \|N^{s_1}P_Nu\|_{l^2_N U^2_{S_\lambda}},\quad \|v\|_{Y^{s_2}}:=\|P_1v\|_{L_x^2L_t^\infty}+\|N^{s_2}P_Nv\|_{l^2_{N\geq 2}V^2_K}.$$
		Then we can show the contraction for the iteration of the initial integral equation. 		
		If $s_2 = -3/4$, $s_1>-3/16$, we can show $\|\eta\mathscr{B}(\partial_x|u|^2)\|_{Y}\lesssim \|u\|_{X_\lambda^{s_1}}$. The argument is also much easier than the case $(s_1,s_2) = (-3/16,-3/4)$. See also \cite{wang2011cauchy}. We omit the details.
		
		\vspace{10pt}
		\textbf{Acknowledgments: The first author is supported in part by the NSFC, grant 12301116. The authors would like to thank Professors Boling Guo and Baoxiang Wang for their invaluable support and encouragement.}
		
		\phantomsection 
		\bibliographystyle{amsplain}
		\addcontentsline{toc}{section}{References}
		\bibliography{reference}
		
		\begin{enumerate}
			\item[] \scriptsize\textsc{Yingzhe Ban: The Graduate School of China Academy of Engineering Physics, Beijing, 100088, P.R. China}
			
			\textit{E-mail address}: \textbf{banyingzhe22@gscaep.ac.cn}
			
			\item[] \scriptsize\textsc{Jie Chen: School of Sciences, Jimei University, Xiamen 361021, P.R. China}
			
			\textit{E-mail address}: \textbf{jiechern@163.com}
			
			\item[] \scriptsize\textsc{Ying Zhang: The Graduate School of China Academy of Engineering Physics, Beijing, 100088, P.R. China}
			
			\textit{E-mail address}: \textbf{zhangying21@gscaep.ac.cn}
		\end{enumerate}
\end{document}